\documentclass{amsart}
\usepackage[utf8]{inputenc}
\usepackage{graphicx}
\usepackage{amssymb}
\usepackage{amsthm}
\usepackage{subcaption}

\usepackage{multirow}
\usepackage{color}
\usepackage{hyperref}

\newtheorem{theorem}{Theorem}

\newtheorem{lemma}[theorem]{Lemma}
\newtheorem{proposition}[theorem]{Proposition}
\newtheorem{corollary}[theorem]{Corollary}
\newtheorem{remark}[theorem]{Remark}
\theoremstyle{definition}
\newtheorem{definition}[theorem]{Definition}

\newcommand{\bd}{\partial}
\newcommand*{\ora}{\protect\overrightarrow}
\newcommand*{\ola}{\protect\overleftarrow}

\author{M. Eudave-Mu\~noz}
\address{ \hskip-\parindent
Mario  Eudave-Mu\~noz\\
Instituto de Matem\'aticas\\
 Universidad Nacional Aut\'onoma de M\'exico\\
  Campus Juriquilla\\ Quer\'etaro, Qro.\\ MEXICO}
 \email{mario@matem.unam.mx}

\author{F. Manjarrez-Guti\'errez}
\address{ \hskip-\parindent
Fabiola Manjarrez-Guti\'errez \\
 Instituto de Matem\'aticas\\
 Universidad Nacional Aut\'onoma de M\'exico\\
 Cuernavaca, Mor.\\
 MEXICO}
\email{fabiola.manjarrez@im.unam.mx}

\author{E. Ram\'{i}rez-Losada}
\address{ \hskip-\parindent
Enrique Ram\'{i}rez-Losada\\
 Centro de Investigaci\'on en Matem\'aticas\\
 Guanajuato, GTO. \\
 MEXICO}
\email{kikis@cimat.mx}

\author{J. Rodr\'{i}guez-Viorato}
\address{ \hskip-\parindent
J. Rodr\'{i}guez-Viorato\\
 Centro de Investigaci\'on en Matem\'aticas\\
 Guanajuato, GTO.\\
 MEXICO}
\email{jesusr@cimat.mx}


\title{Computing genera of satellite tunnel number one knots and torti-rational knots}

\begin{document}
\begin{abstract}
The genus of satellite tunnel number one knots and torti-rational knots is computed using the tools introduced by Floyd and Hatcher \cite{FH}. An implementation of an algorithm is given to compute genus and slopes of minimal genus Seifert surfaces for such knots.
\end{abstract}

\maketitle

\section{Introduction}
A family of knots widely studied is the one known as  $(1,1)$-knots, these are knots which can be put in 1-bridge position with respect a standard torus in $S^3$. This family contains all 2-bridge knots, all satellite tunnel number one knots, and it is contained  in the family of tunnel number one knots. Genus one and genus two $(1,1)$-knots have been classified in \cite{Ma} and \cite{EMR}, respectively. It is natural to ask for a  classification of $(1,1)$-knots of any genus $g$. Such knots are divided  into the satellite and the non-satellite cases. For the non-satellite case we expect to have a description similar to that in \cite{EMR}, as special banding of two $(1,1)$-knots of smaller genus. In the case that the knot is satellite, we need to determine the 4-tuple $\alpha, \beta, p, q$ of the Morimoto-Sakuma construction that produces satellite genus $g$ tunnel number one knots \cite{MS}. The parameters $\alpha, \beta$ describe a rational link $L_{\beta/\alpha}$ and $p,q$ a companion torus knot. For genus $g\geq 3$ a minimal Seifert surface may intersect the companion torus in a non-empty collection of longitudes, hence the surface is broken into two pieces, one piece consists of Seifert surfaces for  the companion torus, the other piece is a surface contained in the neighborhood of the torus knot with one boundary parallel to the satellite knot and boundary components which are slopes on the companion torus. Such a surface defines an essential surface for the link $L_{\beta/\alpha}$, with one boundary parallel to a component of the link and a number of boundary components on the other component.
Floyd and Hatcher \cite{FH} classified essential surfaces for rational links. Later Hoste and Shanahan \cite{HS} classified the slopes of such surfaces. However the calculation of genera of the surfaces is not given there.

We were able to determine that an essential surface $F'$ for a rational link with one boundary on one component of the link and a number of boundary slopes on the other component of the link, arises from at most two minimal edge-paths of the Floyd-Hatcher construction, by means of continued fraction expansions for $\beta/\alpha$. This gives a constructive description of the surfaces and allows to compute genus and slope of the surface, as well as to determine whether or not the surface is a fiber of a fibering over the circle for the link.

Applying these results to satellite tunnel number one knots we obtain the following result:

\begin{theorem}
\label{genus-satellite}
Let $L_{\beta/\alpha}=K_1 \cup K_2$ be the 2-bridge link given by the tunnel number one satellite knot $K(\alpha,\beta, p, q)$. Suppose $lk(K_1, K_2)\neq 0$. Then
\begin{enumerate}
\item If $0\leq \beta \leq \alpha$, $pq\geq 0$ and $[0;2n_1,...,2n_j]$ is the unique continued fraction for $\beta/\alpha$ with $j$ odd, the genus of $F'$ is:
\begin{equation*} 
\frac{1}{2} \Bigg[\displaystyle{\left(-1 + \sum_{ k \; odd}|n_k|\right) } \; (|lk(K_1,K_2)pq| - 1) + (j + 1) - (\vert lk(K_1, K_2)\vert+1)\Bigg]
\end{equation*}
where $k\in \{1, ..., j\}$
\item If $0\leq \beta \leq \alpha$, $pq\leq 0$ and $[1;2m_1,...,2m_i]$ is the unique continued fraction for $\beta/\alpha$ with $i$ odd, the genus of $F'$ is:
\begin{equation*} 
\frac{1}{2} \Bigg[\displaystyle{\left(-1 + \sum_{ h \; odd}|m_h|\right) } \; (|lk(K_1,K_2)pq| - 1) + (i + 1) - (\vert lk(K_1, K_2)\vert+1)\Bigg]
\end{equation*}
where $h\in \{1, ..., i\}$
\end{enumerate}
\end{theorem}

\begin{corollary}
Let $K=K(\alpha,\beta, p, q)$ be a tunnel number one satellite knot such that $lk(K_1, K_2)\neq 0$. Then the genus of $K$ is:
\begin{equation*}
g(K)= g(F')+ \vert lk(K_1, K_2)\vert \frac{(\vert p\vert-1)(\vert q \vert-1)}{2}
\end{equation*}

\end{corollary}

It is worth mentioning that Hirasawa and Murasugi \cite{HM} obtained similar results using the Alexander polynomial.

We can also apply our technique to compute the genus of torti-rational knots, which are obtained from a rational link by performing $r$-Dehn twists along one component of the link.

\begin{theorem}
\label{genus-torti}
Let $K(\beta/\alpha; r)$ be a torti-rational knot and $F$ a minimal genus Seifert surface for it. Suppose that $lk(K_1, K_2)\neq 0$. Then:
\begin{enumerate}
\item If $r > 1$ and $[1;2m_1,...,2m_i]$  is the unique continued fraction for $\beta/\alpha$ with $i$ odd, the genus of $F$ is:
\begin{equation*} 
\frac{1}{2} \Bigg[\displaystyle{\left(-1 + \sum_{ h \; odd}|m_h|\right) } \; (|lk(K_1,K_2)r| - 1) + (i + 1) - (\vert lk(K_1, K_2)\vert+1)\Bigg]
\end{equation*}
where $h\in \{1, ..., i\}$

\item If $r < 1$ and $[0;2n_1,...,2n_j]$ is the unique continued fraction for $\beta/\alpha$ with $j$ odd, the genus of $F$ is:
\begin{equation*} 
\frac{1}{2} \Bigg[\displaystyle{\left(-1 + \sum_{ k \; odd}|n_k|\right) } \; (|lk(K_1,K_2)r| - 1) + (j + 1) - (\vert lk(K_1, K_2)\vert+1)\Bigg]
\end{equation*}
where $k\in \{1, ..., j\}$

\item If $\vert r \vert  > 1 $ and $\vert lk(K_1,K_2)\vert > 1$. Let  
$[s;2r_1,...,2r_k]$  be the continued fraction expansion for $\beta/\alpha$ with $s=0$ or $1$ such that  $k\geq 3$ and  $\vert r_t \vert \geq 2$ for all $t$. The genus of $F$ is:
\begin{equation*}
    1+ \frac{(\vert lk(K_1, K_2)\vert+1)(k-3)}{4}
\end{equation*}

\item If $\vert r \vert =1$ and $\vert lk(K_1,K_2)\vert = 1$ and $[0;2n_1,...,2n_j]$ and $[1;2m_1,...,2m_i]$ are the continued fraction for $\beta/\alpha$ with $j,i$ odd. The genus of $F$ is:
\begin{equation*}
    min \Bigg(\frac{i-1}{4}, \frac{j-1}{4}\Bigg)
\end{equation*}

\end{enumerate}
\end{theorem}

For the case that $lk(K_1, K_2)=0$ for both cases we prove:

\begin{theorem}
When the linking number is zero, the genus of a satellite tunnel number one knot $K(\alpha, \beta, p, q)$ is one half the wrapping number of $K_2$ in $E(K_1)$. Moreover, if 
$[s;2r_1,...,2r_k]$ is the continued fraction expansion for $\beta/\alpha$ with $s=0$ or $1$ such that  $k$ odd, the genus of $K(\alpha, \beta, p, q)$ is $\Sigma \vert r_i\vert$. The same is true for a torti-rational knot.
\end{theorem}

Theorems \ref{genus-satellite} and \ref{genus-torti} required a decomposition of $\beta/\alpha$ as a continued fraction and some computations. We have written an algorithm that receives as inputs $\alpha, \beta, p, q, r$ and outputs genus, slopes and number of boundary components for the surface, at some cases it can determined the fiberedness of the knot.

Our algorithm is based on that given by Hoste and Shanahan. We found a fault for rationals $\beta/\alpha > 1/2$, thus it was necessary to reprogram this algorithm to compute the paths and to incorporate computations of genus, slopes and number of boundary components.  Our modification of their algorithm can be found at \href{https://github.com/viorato/compute\_rational\_links\_genus}{https://github.com/viorato/compute\_rational\_links\_genus}.

In Section \ref{preliminaries} we review the concepts from Floyd-Hatcher which are necessary to develop our techniques. In Section \ref{general results} we state the basic results that allow to describe the specific type of edge-paths associated to the surfaces of our interest. Using continued fraction expansions for $\beta/\alpha$ we compute genus and slopes for the surfaces in Section \ref{fractions and genus}. We revisite Floyd-Hatcher to give their criteria for a surface to be a fiber of a fibering for a rational link and give a criteria in terms of the continued fraction expansions for our surfaces to be fibers in Section \ref{fiberings}. Finally in Section \ref{applications} we compute the genus for satellite tunnel one knots and for torti-rational knots.

\section{Preliminaries}
\label{preliminaries}

\subsection{The diagram of slope system in the four puncture sphere}
A 2-bridge link $L_{\beta/\alpha}$ is represented by a rational number $\beta/\alpha$. We may suppose $0 <  \beta  < \alpha$, $\alpha$ even, and $gcd(\alpha,\beta)=1$. 
 We say that a surface in $S^3-L_{\beta/\alpha}$ is essential if it is incompressible, $\bd$-incompressible and not boundary parallel.
The main idea of Floyd and Hatcher's \cite{FH}
construction is to associate to an essential surface in $S^3-L_{\beta/\alpha}$ an edge-path from $1/0$ to $\beta/\alpha$ in the Diagram $D_t$,  $t\in [0,\infty]$, shown in Figure \ref{fig:AllDt}.

The diagram $D_1$ is an embedded graph on the upper half plane $\mathbb{H}$ with the real line $\mathbb{R}$ and the point at infinite $1/0$. Its vertices are the rational points in $\mathbb{R}\cup \{1/0\}$, and its edges are hyperbolic lines in the upper half model of $\mathbb{H}$ joining two vertices $a/c$, $b/d$, $(a,b,c,d \in \mathbb{Z})$ if and only if $ad-bc=\pm 1$. These lines  are the edges of ideal triangles in $\mathbb{H}$, and $PSL_2(\mathbb{Z})$ is the group of orientation-preserving symmetries of this ideal triangulation. The diagram $D_1$ is transformed onto the Poincar\'e disk model by $-\frac{z-\frac{1+i}{2}}{z-\frac{1-i}{2}}$, see Fig.\ref{fig:AllDt}. Let $G \subset PSL_2(\mathbb{Z})$ be the subgroup of M\"obius transformations $(az+b)/(cz+d)$ with $c$ even. Its fundamental domain is the triangle $\langle 1/0,0/1,1/1 \rangle$. Consider the ideal quadrilateral $Q = \langle 1/0, 0/1, 1/2, 1/1 \rangle$. The $G$-images of this quadrilateral tessellate $\mathbb{H}$. We form the diagram $D_0$ from $D_1$ by deleting the $G$-orbit of the diagonal $\langle 0/1,1/1 \rangle$ of $\langle 1/0, 0/1, 1/2, 1/1 \rangle$ and adding the $G$-orbit of the opposite diagonal $\langle 1/0,1/2 \rangle$. The diagram $D_t$, $0<t<\infty$, $t\not= 1$, is obtained from $D_1$ by deleting the diagonal $\langle 0/1,1/1 \rangle$ in each quadrilateral $Q$, and adding a small rectangle having a vertex in the interior of each edge of $Q$ so that $g(D_t)=D_t$ for $g\in G$. The edges of $D_t$ fall into four $G$-orbits, labelled $A$, $B$, $C$, $D$.
\begin{figure}
    \centering
    \begin{subfigure}[b]{0.4\textwidth}
        \includegraphics[width=\textwidth]{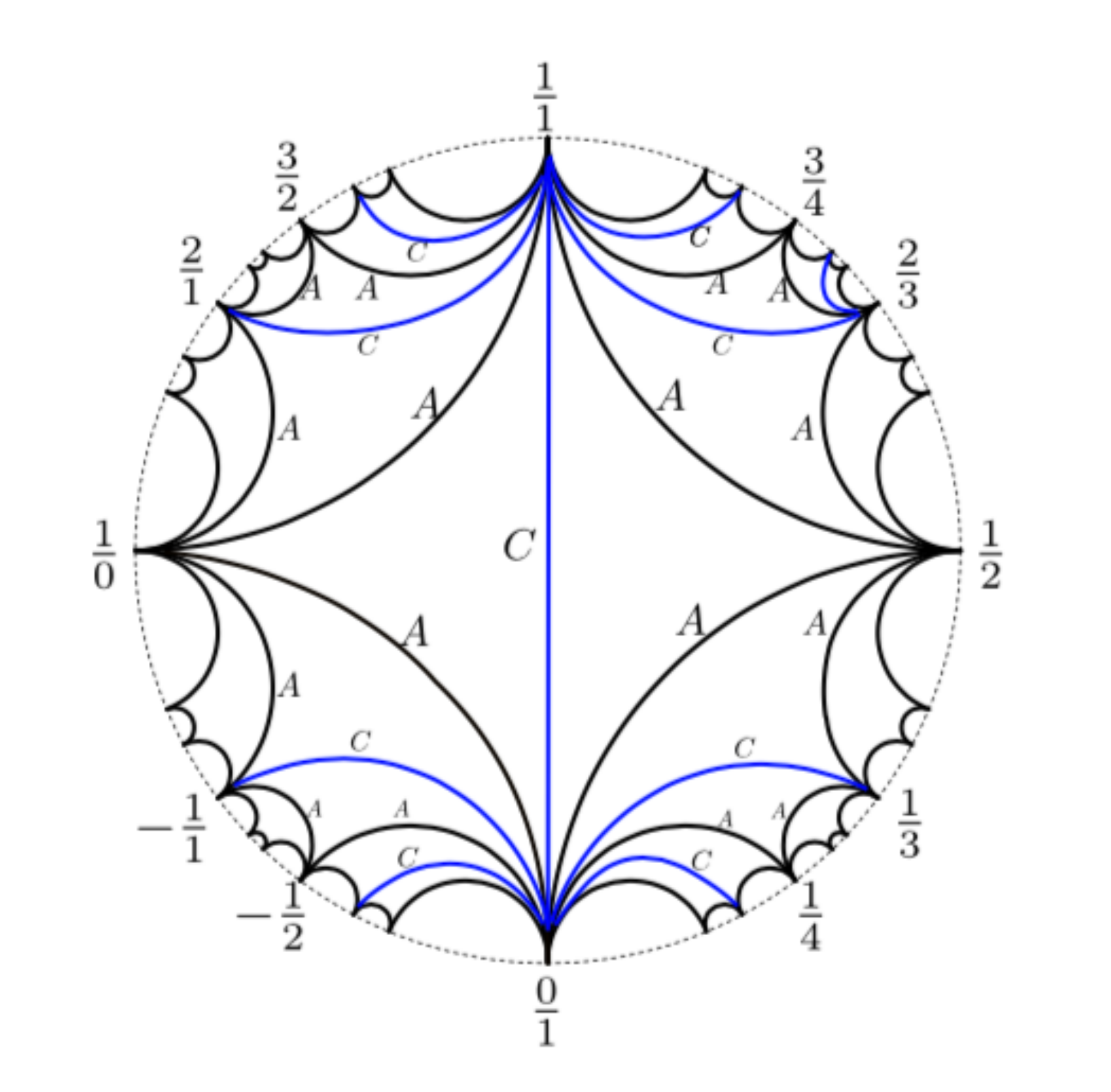}
        \caption{Diagram $D_1$}
        \label{fig:D1}
    \end{subfigure}
    \vfill
        \begin{subfigure}[b]{0.4\textwidth}
        \includegraphics[width=\textwidth]{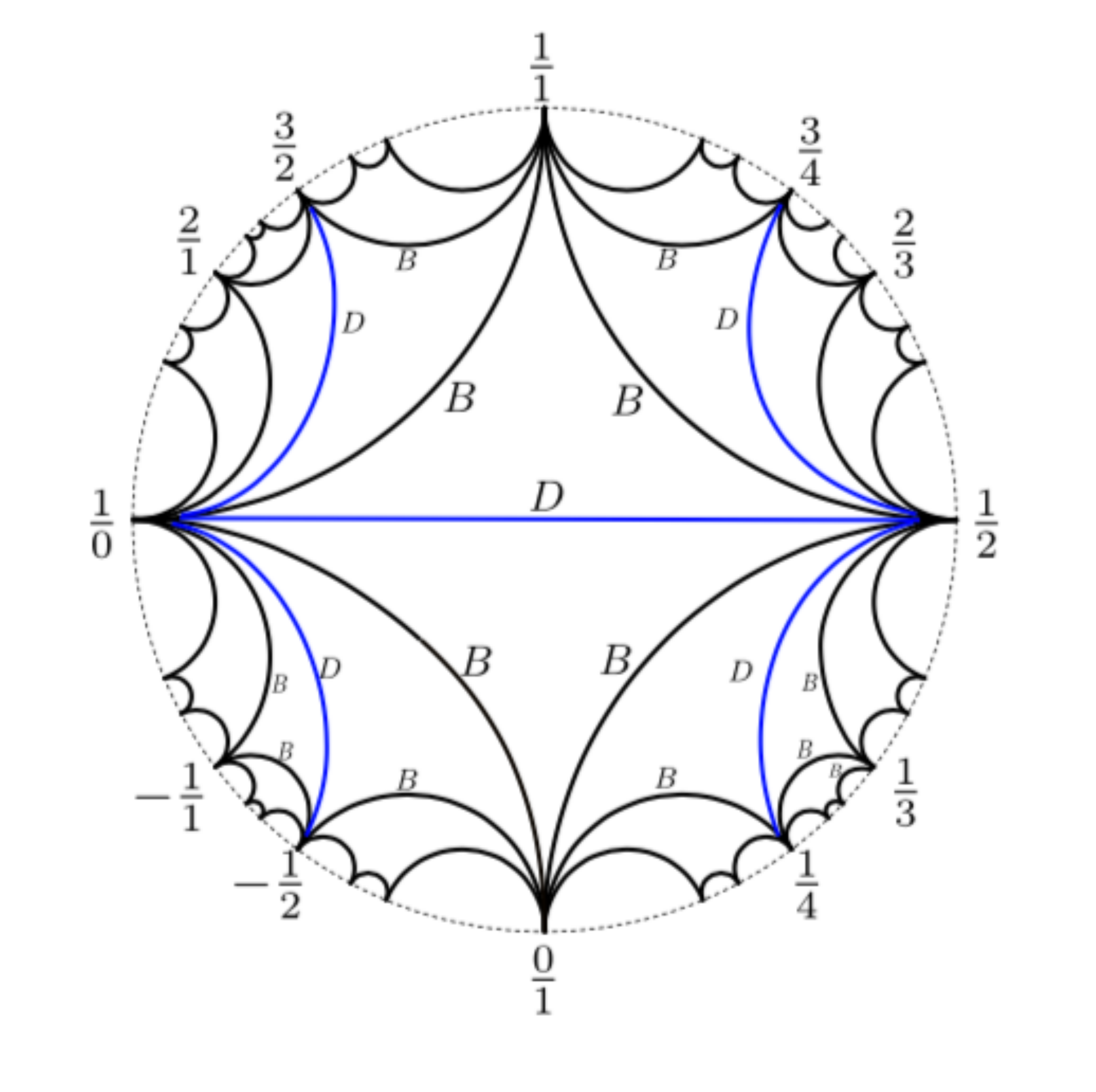}
        \caption{Diagram $D_0=D_\infty$}
        \label{fig:D0}
    \end{subfigure}
    \vfill
     \begin{subfigure}[b]{0.4\textwidth}
        \includegraphics[width=\textwidth]{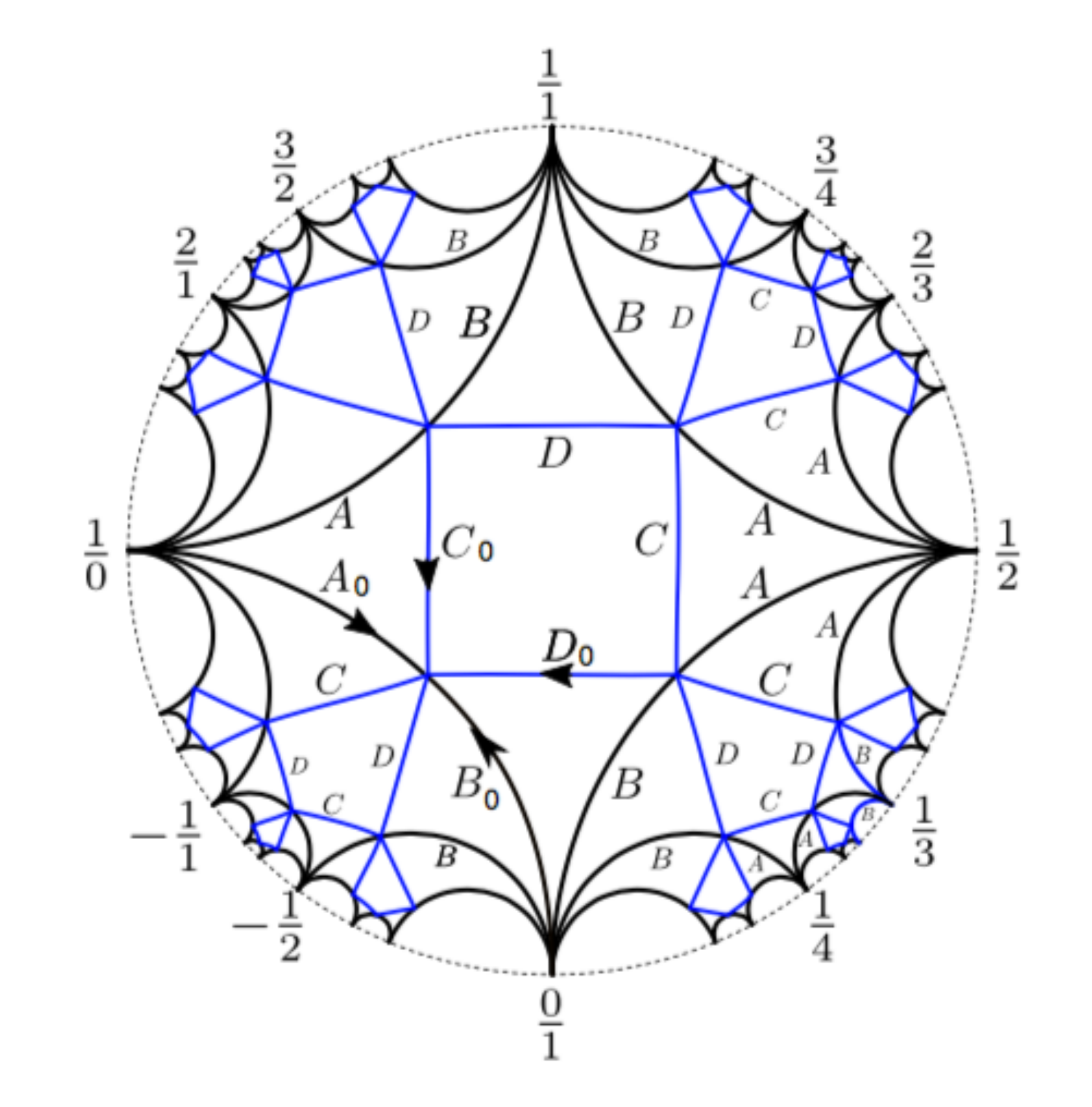}
       \caption{Diagram $D_t$, $t \neq 0, 1, \infty$}
        \label{fig:Dt}         

    \end{subfigure}
    \caption{}
    \label{fig:AllDt}
\end{figure}

\begin{remark}
\label{Dtcolapsado}
As $t$ approaches to 0 and 1, the inscribed rectangle collapses to the diagonals $\langle 1/0, 1/2 \rangle$ or to the diagonal $\langle 0/1, 1/1 \rangle$, respectively. See Figure \ref{fig:colap}.
\end{remark}

\begin{figure}[ht]
\includegraphics[width=5cm]{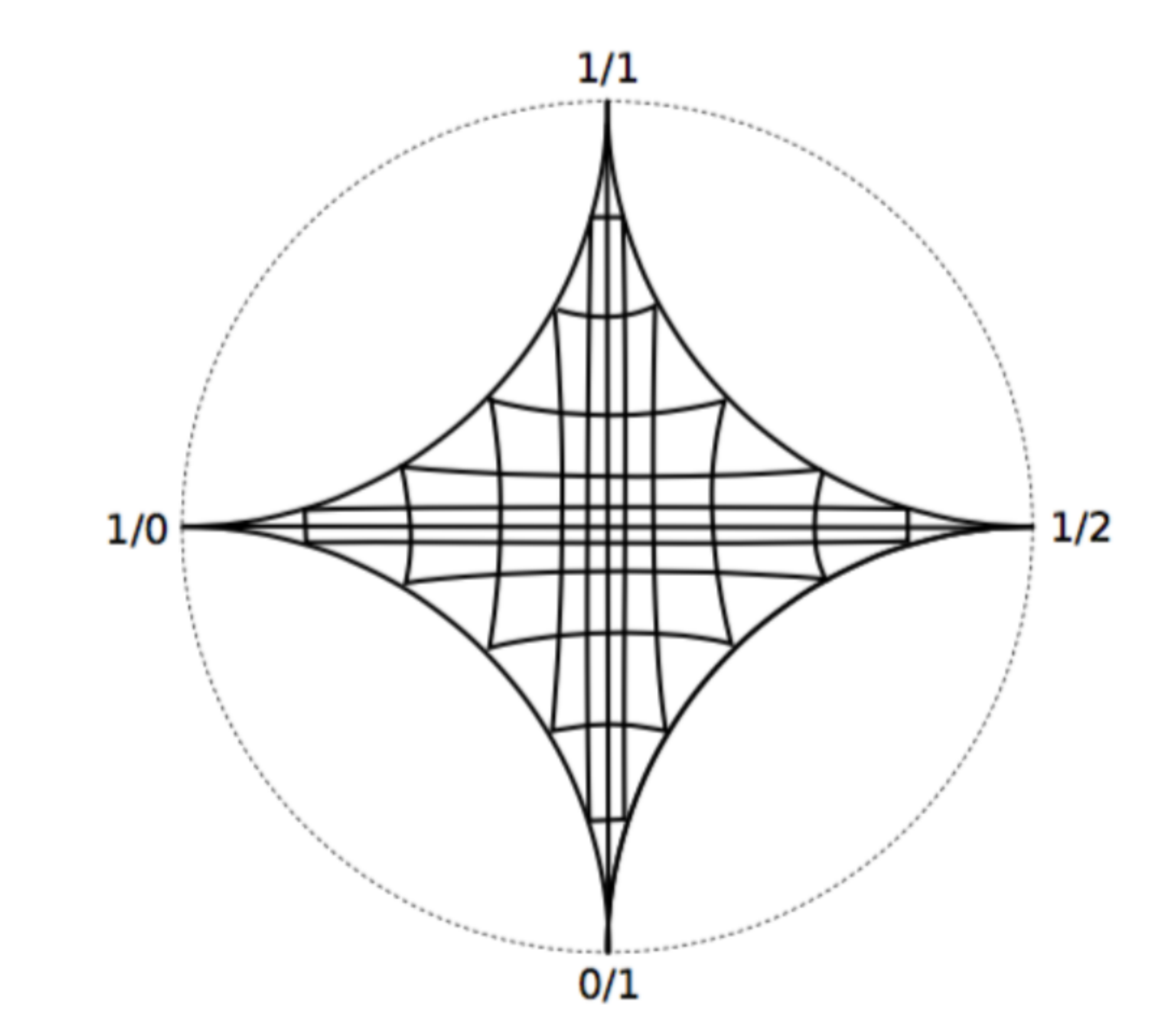}
\caption{Collapsing a $D_t$ diagram.}
\label{fig:colap}
\end{figure}

For a given reduced rational number $\beta /\alpha$, let $\gamma$ denote an oriented edge-path from $1/0$ to $\beta/\alpha$ in $D_t$ with $0\leq t \leq \infty$.

\begin{definition} An edge-path $\gamma$ is called minimal if no two consecutive edges in $\gamma$ lie on the boundary of the same triangle face or rectangle face in $D_t$.
\end{definition}

Then for every minimal edge-path $\gamma$ in $D_t$, Floyd-Hatcher construct a corresponding branched surface $\Sigma_\gamma$. Four basic branched surfaces, $\Sigma_A$, $\Sigma_B$, $\Sigma_C$, and $\Sigma_D$ are assigned to the labelled edges. See figure \ref{fig:branch}. 
\begin{figure}[ht]
\includegraphics[width=10cm]{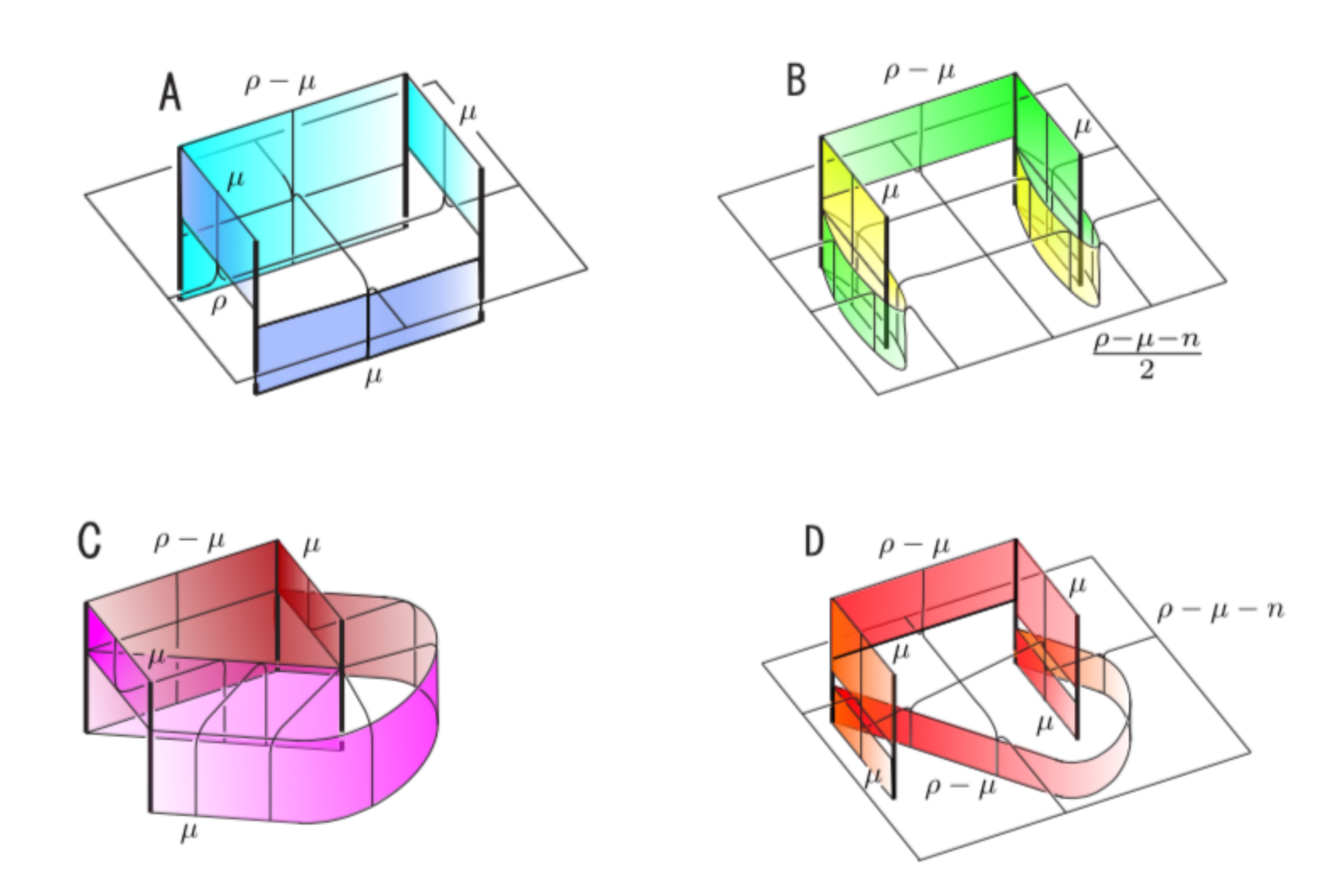}
\caption{Branched pieces for edges $A, B, C$ and $D$}
\label{fig:branch}
\end{figure}

We regard $S^3$ as the two point compactification of $S^2\times\mathbb{R}$ and we place the link $L_{\beta/\alpha}\subset S^2 \times I$ so that it meets $S^2 \times \{0\}$ and $S^2 \times \{1\}$ each in two arcs and each intermediate level in four
points. We think of each level $S^2 \times \{r\}$ as the quotient $\mathbb{R}^2 /\Gamma$, where $\Gamma$ is the group generated by $180^\circ$ rotations of $\mathbb{R}^2$ about the integer lattice points $\mathbb{Z}^2$. The four points of the link at each
intermediate level are precisely the four points of $\mathbb{Z}^2/\Gamma$. The two arcs at level $r=1$ have slope $\beta/\alpha$ and those arcs at level $r=0$ have slope $1/0$. $PSL_2 (\mathbb{Z})$ acts linearly on the level sphere $S^2 \times \{r\} = \mathbb{R}^2/\Gamma$, leaving $\mathbb{Z}^2/\Gamma$ invariant.

The vertices of the diagrams $D_1, D_0= D_{\infty}, D_t$ correspond to the slopes of arcs in the level spheres.

Let $e_1, \dots, e_k$ be the sequence of edges of a minimal edge-path $\gamma$. An edge $e_i$ is the image of one of the four edges, $A_0$, $B_0$, $C_0$, $D_0$, see Figure \ref{fig:Dt}, under a unique $g_i \in G$. To get $\Sigma_\gamma$ we first apply $g_i\times id_I$ to the appropriated surface $\Sigma_A$, $\Sigma_B$, $\Sigma_C$, or $\Sigma_D$, and then scale vertically into interval $[(i-1)/k, i/k]$.

Finally, a surface carried by one of the branched surfaces $\Sigma_\gamma$ is determined by $\mu$ and $\rho$, the numbers of sheets of the surface along each component of the 2-bridge link $L_{\beta/\alpha}$, and by how the surface branches in each segment $\Sigma_A$, $\Sigma_B$, $\Sigma_C$, or $\Sigma_D$ of $\Sigma_\gamma$. We set $t=\mu/\rho$, which is the subscript of $D_t$ .

Floyd and Hatcher proved that every essential surface in $S^3-L_{\beta/\alpha}$ is carried by some branched surface corresponding to a minimal edge-path from $1/0$ to $\beta/\alpha$ in $D_t$, and, conversely, an orientable surface carried by such a branched surface is essential.

A branched surface may carry non-orientable surfaces. Moreover, there may be an essential non-orientable surface which is not carried by any branched surface.

There is a unique finite sequence of quadrilaterals $Q_{\beta/\alpha}$ such that the first one contains the vertex $1/0$, the last one contains the vertex $\beta/\alpha$ and every pair of consecutive ones intersects in a single edge. 

\begin{remark}
\label{sillaA}
In a $D_t$ diagram with $t\neq 0, \infty$, the first and the last edges in any edge-path are of type $A$.
\end{remark}

\subsection{Edge-paths and essential saddles}
\label{sectionpathsandsaddles}

Let $S\subset S^3- L_{\beta/\alpha}$ be a compact orientable essential surface with boundary on $L_{\beta/\alpha}\subset S^2 \times I \subset S^3$.

We may isotope $S$ so that:
\begin{enumerate}
\item Each component of $\partial S$ is either a meridian of $L$ in $S^2 \times (0,1)$, or is transverse to all meridians of $L$.
\item $S$ is transverse to $S^2 \times \partial I$ and lies in $S^2 \times I$ near $L \cap (S^2 \times \partial I)$.
\item The projection $S\cap (S^2\times I)\rightarrow I$ is a Morse function with all its critical points in the interior of $S$.
\end{enumerate}
A transverse intersection $S\cap S^2_r$, $S_r^2= S^2\times r$, for $0<r<1$, can contain no arcs which are peripheral in $S_r^2- L_{\beta/\alpha}$, in view of $(1)$ and the $\partial$-incompressibility of $S$.
As $r$ varies from 0 to 1, the point $\lambda_r \in D_t$ can change only at critical levels of the projection $S\cap(S^2\times I)\rightarrow I$, in fact, only at saddles. A saddle where $\lambda_r$ changes we call an essential saddle. So we obtain a finite sequence of $\lambda_r'$s, say $\lambda_0, ..., \lambda_k$, with $\lambda_{i+1}\neq \lambda_{i}$ for all $i$. By $(2)$, $\lambda_0$ is the vertex $1/0$ of $D_t$ and $\lambda_k$ is the vertex $\beta/\alpha$.

We can isotope $S$ to lie in $S^2 \times I$ and have all its critical points essential saddles, and also still satisfy $(1)-(3)$ above, see section 7 of \cite{FH}.

The possibilities, up to level-preserving isotopy, for an essential saddle corresponding to a segment $\langle \lambda_i, \lambda_{i+1} \rangle$ on an $A-, B-, C-$ or $D-$type edge of $D_t$ are shown in Figure \ref{fig:Saddletypes}. The two leftmost vertices depict $K_2 \cap S_r^2$ and the rightmost vertices depict $K_1 \cap S_r^2$.

\begin{figure}[ht]
    \centering
    \begin{subfigure}[b]{0.5\textwidth}
        \includegraphics[width=\textwidth]{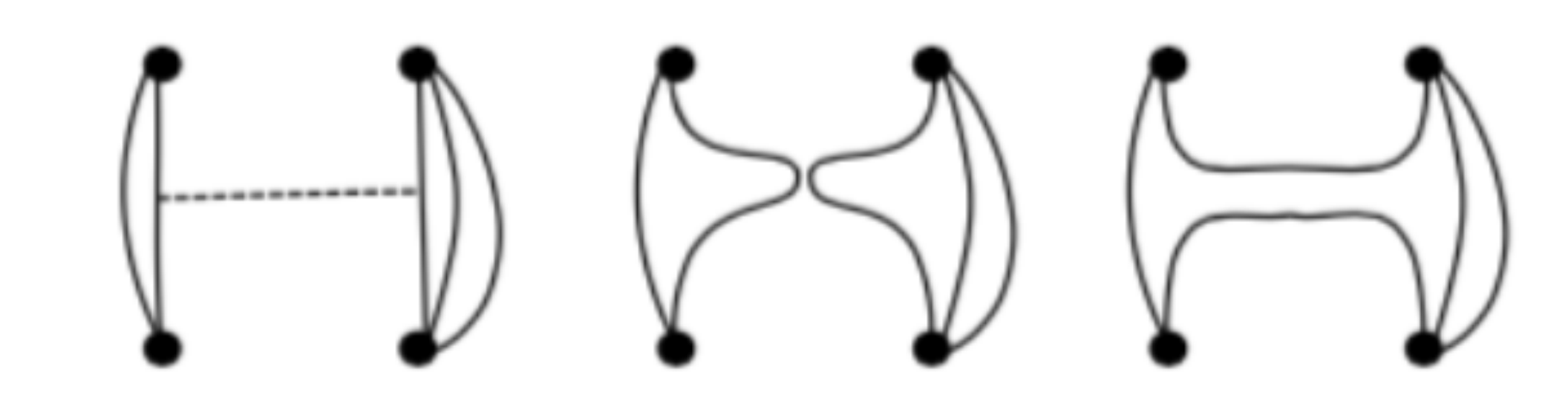}
        \caption{A-type saddle}
        \label{fig:A}
    \end{subfigure}
    \vfill
   
    \begin{subfigure}[b]{0.5\textwidth}
        \includegraphics[width=\textwidth]{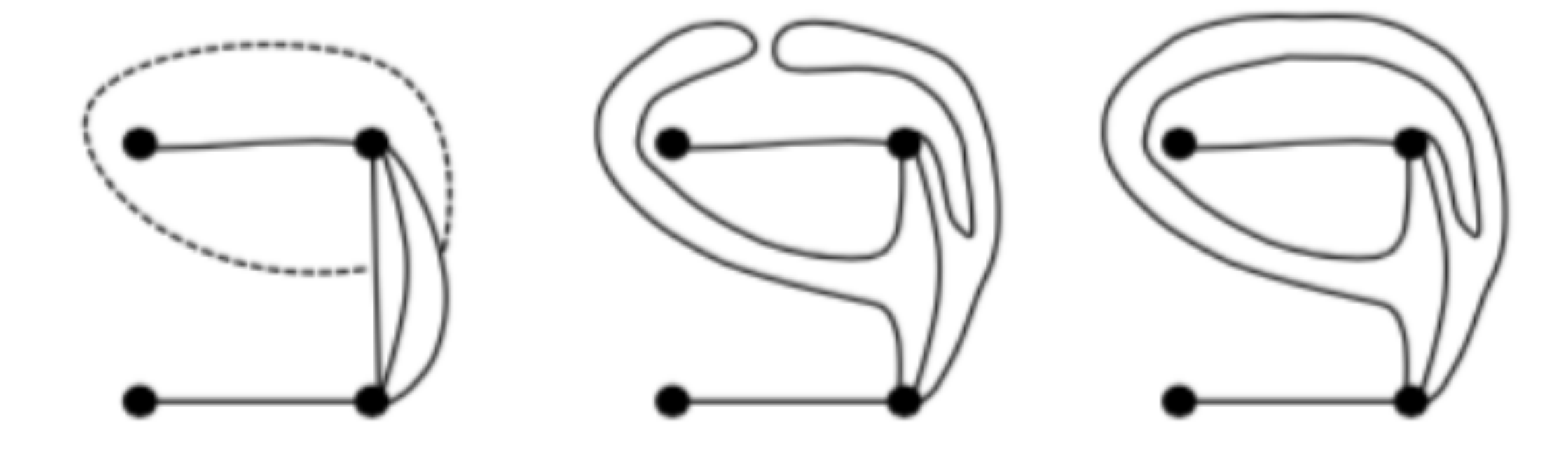}
        \caption{B-type saddle}
        \label{fig:B}
    \end{subfigure}
    \vfill
    
    \begin{subfigure}[b]{0.5\textwidth}
        \includegraphics[width=\textwidth]{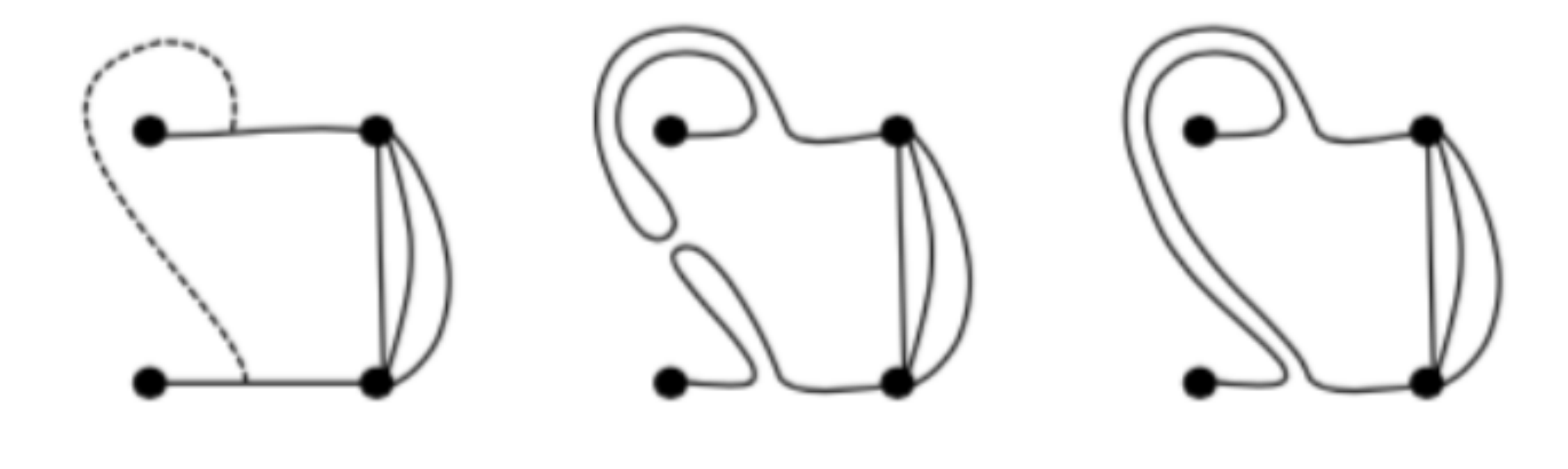}
       \caption{C-type saddle}
        \label{fig:C}         

    \end{subfigure}
    \vfill
 \begin{subfigure}[b]{0.5\textwidth}
        \includegraphics[width=\textwidth]{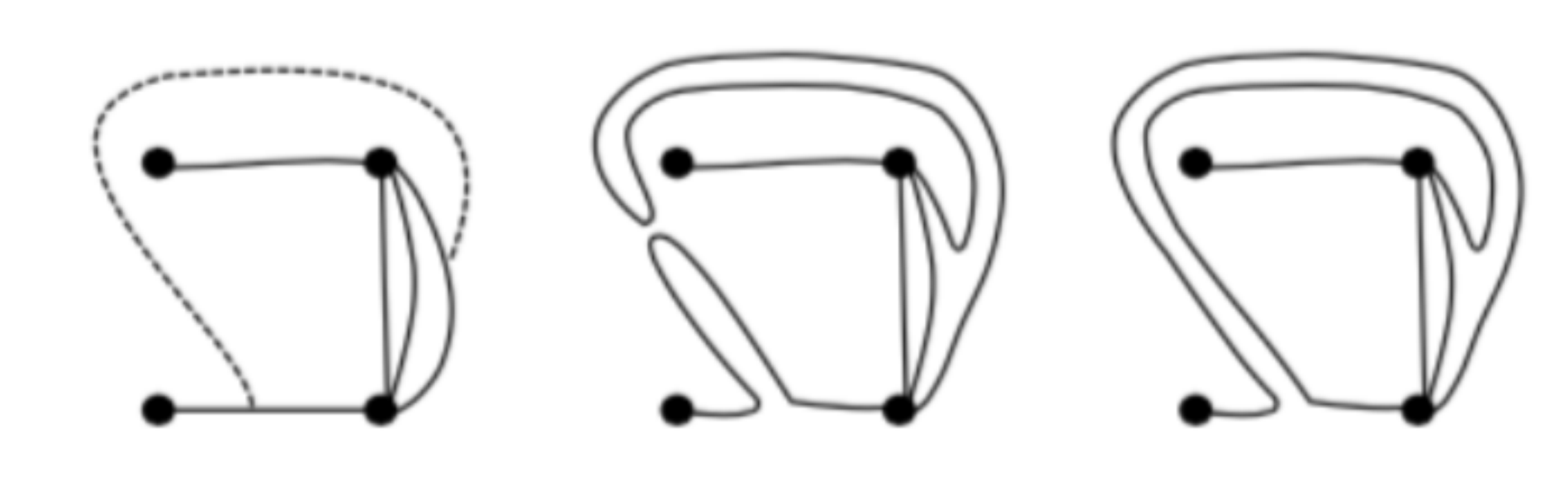}
       \caption{D-type saddle}
        \label{fig:D}         

    \end{subfigure}
    \caption{Saddle types: The two leftmost vertices depict $K_2 \cap S_r^2$ and the rightmost vertices depict $K_1 \cap S_r^2$.}
    \label{fig:Saddletypes}
\end{figure}

The corresponding saddle to an $A-, B-, C-$ or $D-$type edge of $D_t$, will be called an $A-, B-, C-$ or $D-$type saddle, respectively.

\section{General results}
\label{general results}

Let $L_{\beta/\alpha}=K_1 \cup K_2$ be a rational link embedded in $S^3$ and let $S\subset S^3- L_{\beta/\alpha}$ be a connected, compact, essential and orientable surface, both as in Section \ref{sectionpathsandsaddles}. Assume that $S$ has $n$-boundary components in $K_1$, which are non-meridional and $n\neq 0$, i.e,  $\mu$ is a multiple of $n$; and has one boundary component parallel to $K_2$, i.e, $\rho=1$. Let us denote by $\partial_i S$ the set of boundary components of $\partial N(K_i)\cap S$, for $i=1,2$. Observe that $\partial_2 S$ consists only of one curve whose slope is an integer; and $\partial_1 S$ of $n$ parallel curves with slope $p/q$, with respect a meridian and preferred longitude in each component of the link.  We denote the linking number of  $L_{\beta/\alpha}=K_1 \cup K_2$ by $lk(K_1, K_2)$.

In the following lemmas we will determine the saddle types corresponding to a minimal edge-path associated to $S$. Since there is a bijective correspondence between edges, saddles and pieces of branched surfaces, the results can be applied to the three concepts.

Since $\mu/\rho \neq 0, \infty$ and by Remark \ref{sillaA} the first and last saddles are of type $A$. By Lemma 7.1 and Figure 7.2 of \cite{FH}, we have the following statement:

\begin{lemma}
\label{lemma:saddles-come-in-groups}
Suppose that $\mu\neq 1$.
\begin{enumerate}
\item $B-$type saddles come in groups of $(\mu-1)/2$ saddles.
\item $D-$type saddles come in groups of $(\mu-1)$ saddles.
\end{enumerate}
\end{lemma}

Next we will prove that only edges of type $A, B$ and $D$ can occur. Choosing an orientation for $S$ will induce an orientation on the boundary components of $S$ and on the arcs of $S\cap S^2_{\epsilon}$, for $\epsilon$ before the first $A-$type saddle; choose one. When two arcs are being fused by a saddle, in a small neighborhood before the fusion occurs, we see two small arcs with opposite orientations. 

\begin{lemma}\label{lemma:no-c-type}
 There are no $C$-type saddles.
\end{lemma}
\begin{proof}

At the first level, $S^2_{\epsilon}$, there is only one arc of $S\cap S^2_{\epsilon}$ connecting the vertices of $K_2\cap S^2_{\epsilon}$. This implies that, in a small neighborhood around one of the vertices of $K_2\cap S^2_{\epsilon}$, we see only one arc pointing out and around the other vertex we see only one arc pointing in; we see opposite orientations around these vertices. This property must be preserved for all the different levels $S^2_{r}$.

Now, if a $C$-type saddle exists then, after a $G$ transformation, it looks like in Fig. \ref{fig:C} . But that will imply that the orientations around the vertices $K_2\cap S^2_{r}$, at some $r$, are no longer opposite.

\end{proof}

One crucial object that we used on the proof of Lemma \ref{lemma:no-c-type} and that we will use is the orientation of $S \cap S^2_{\epsilon}$ around a small neighborhood of a vertex. Once that we orientate $S$, it induces an orientation on the arcs $S \cap S^2_{\epsilon}$ around a vertex, we can assign a $+1$ to each arc pointing out and a $-1$ to an arc pointing in. We can then compute the sum of the signs around a vertex $v$; we denote it as $\Sigma_v$. Observe that $\Sigma_v$ is independent of the level $S^2_{r}$ and it reverses its sign if we change the orientation of $S$. So, $|\Sigma_v|$ is a constant that is independent of the level $S^2_{r}$ and the orientation of $S$.

\begin{lemma}\label{lemma:signsum-and-linking}
If the boundary slope of $\partial_1 S$ is of the form $p/q$ with $p, q \in \mathbb{Z}-\{0\}$. Then $|p/q \Sigma_v| =  | lk(K_1, K_2)|$  for each vertex $v$ in $K_1\cap S^2_{\epsilon}$
\end{lemma}

\begin{proof}
By definition, we can compute $|\Sigma_v|$ around any  meridian $m$ of $K_1$. And this can be done by computing the intersections with signs of $\partial_1 S$ and $m$. As the slope of $\partial_1 S$ is $p/q$, each boundary component of $\partial_1 S$ intersects $m$ exactly $q$ times, let $n_+$ be the number of the components intersecting positively $m$ and $n_-$ the number of components which intersect $m$ negatively, then $\Sigma_v = q(n_+-n_-)$.

Now, we only need to prove that $p(n_+ - n_-)= lk(K_1, K_2)$. This can easily be seen by observing that $S$ represents an equivalence between $\partial_2 S = k \mu_2 + \lambda_2$ and $\partial_1 S = (n_+ - n_-)(p\mu_1 + q\lambda _1)$ on $H_1(E(L_{\beta/\alpha}))$ and later combine it with the relations $\lambda_1 = lk(K_1, K_2)\mu_2$ and $\lambda_2 = lk(K_1, K_2)\mu_1$.

\end{proof}

From the previous proof, it seems that we could get rid of the absolute values from the statement. But the problem is that our definition of $\Sigma_v$ has an ambiguity on its sign, it is possible to avoid it by being more specific on its definition, but we wouldn't win much; it is more convenient to use and compute $|\Sigma_v|$.

\begin{lemma}
\label{lemma:sillas}
Suppose that $\mu>1$, and let $S$ be a surface given by an edge-path in $D_t$.
\begin{enumerate}
\item If there is a $B-$type saddle, then $|\Sigma_v| = 1$ for all $v$ in $ K_1  \cap S^2_{\epsilon}$. Moreover, each boundary component of $S$ in $K_1$ is longitudinal and $\mu = n$. 
\item If there is a $D-$type saddle, then  $|\Sigma_v| = \mu$ for all $v$ in $ K_1  \cap S^2_{\epsilon}$. Moreover, all boundary components of $S$ have the same orientation.
\end{enumerate}
\end{lemma}
\begin{proof}

(1) By Lemma \ref{lemma:saddles-come-in-groups} the number of arcs in $S\cap S^2_{\epsilon}$ joining the components of $K_1\cap S^2_{\epsilon}$ is odd. Before a $B-$type saddle appears, there must be an $A-$type saddle. After passing it, we see an even number of arcs joining the components of $K_1\cap S^2_{\epsilon}$. In order to perform a $B-$type saddle, two arcs of the same slope must be joined; thus their orientation are opposite. So, all the arcs joining the components of $K_1\cap S^2_{\epsilon}$ can be paired together on opposite orientation pairs. This implies that $|\Sigma_v| = 1$ for each vertex $v \in K_1\cap S^2_{\epsilon}$. 

By Lemma  \ref{lemma:signsum-and-linking}, we have that the slope $\partial_1 S = p/q $ is equal to $lk(K_1, K_2)$, hence  $\partial_1 S$ is an integer (its components are longitudinal). 

(2) After a $G$ transformation, a $D-$type saddle looks like in Fig. \ref{fig:D}. When performing a $D-$type saddle, the configuration of arcs that we obtain contains two arcs of slope zero whose orientations coincide with the one on the previous arcs of slope zero. This occurs every time we perform a $D-$type saddle, and by Lemma \ref{lemma:saddles-come-in-groups} this happens $\mu -1$ times, thus the arcs $S\cap S^2_{\epsilon}$ joining the components of $K_1\cap S^2_{\epsilon}$ have the same orientation. So, $|\Sigma_v| = \mu$.

\end{proof}

An immediate consequence of Lemmas \ref{lemma:sillas} and \ref{lemma:signsum-and-linking} is the following.

\begin{corollary}
If there is a $B-$type saddle and if the boundary slope of $S\cap \partial N(K_1)$ equals $1/r$ then $\vert lk(K_1, K_2)\vert=1$ and $r=1$
\end{corollary}

Summarizing we have:

\begin{corollary}
\label{coro:AB-and-AD-edgepaths}
Let $L_{\beta/\alpha}=K_1 \cup K_2$ be a rational link embedded in $S^3$ and let $S\subset S^3- L_{\beta/\alpha}$ be a connected, compact, essential and orientable surface. Assume that $S$ has $n$-boundary components in $K_1$, which are non-meridional and $n\neq 0$, i.e,  $\mu$ is a multiple of $n$; and has one boundary component parallel to $K_2$, i.e, $\rho=1$. Then the sequence of saddles consists only of $A-$ and $B-$type saddles; or only of $A-$ and $D-$type saddles; otherwise we will have $1 = |\Sigma_v| = \mu$ but $\mu > 1$.

\end{corollary}

\begin{definition}
We will use the notation $AB-$edge-path to refer to an edge-path consisting of only $A-$ and $B-$type saddles. Similarly we use the notation $AD$- and $A-$edge-path.
\end{definition}

When $\mu=1$ the sequence of saddles is an $A-$edge-path. In this case $\mu/\rho=1$, thus the corresponding edge-path lies in the $D_1$ diagram and there are no $C-$types saddles by Lemma \ref{lemma:no-c-type}.

If $S$ is oriented surface with $\mu > 1$ then it comes from an $AD-$edge-path. Nevertheless, not all $AD-$edge-path correspond to an orientable surface. 

For instance, consider the edge-path $\langle 1/0, 0/1 \rangle$, $\langle 0/1, 1/2  \rangle$, $\langle 1/2, 1/3 \rangle$, $\langle 1/3, 3/8 \rangle$, the corresponding sequence of saddles is $ADAADA$, see Figure \ref{fig:pathnorientable}. In Figure \ref{fig:sillasnorientable} we shown the first part of the saddle sequence (recall that we are using $\mu -1$ type $D$ saddle). Observe that passing to the third saddle of type $A$ gives rise a nonorientable surface.

\begin{figure}[ht]
\includegraphics[width=5cm]{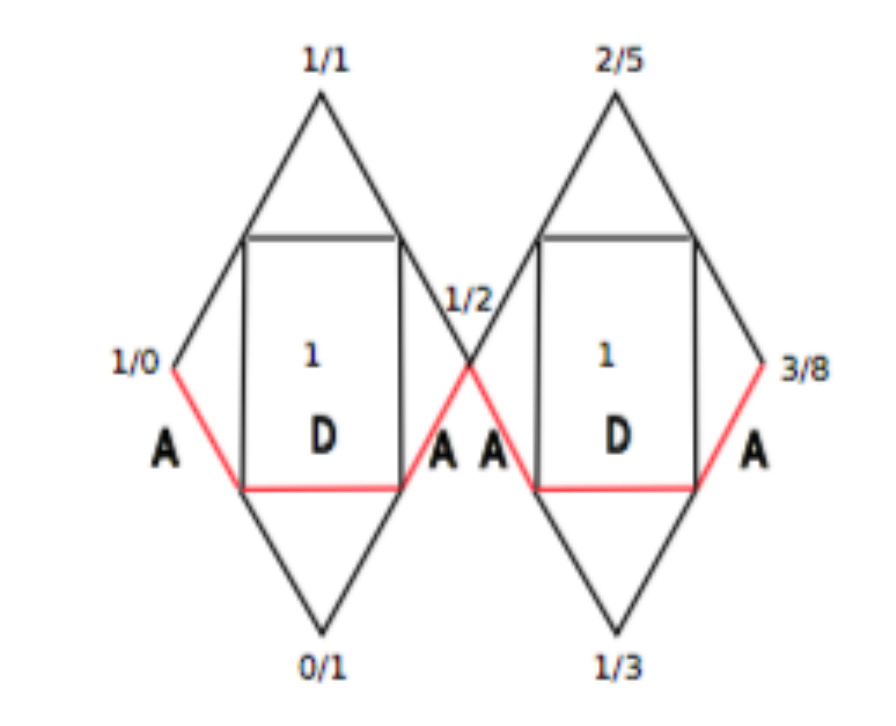}
\caption{An edge-path from $\frac{1}{0}$ to $\frac{3}{8}$}
\label{fig:pathnorientable}
\end{figure}

\begin{figure}[ht]

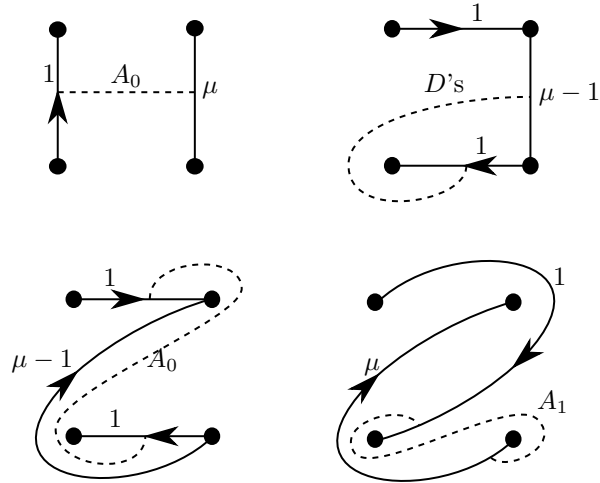
\caption{A-type saddles for the edge-path from $\frac{1}{0}$ to $\frac{3}{8}$}
\label{fig:sillasnorientable}
\end{figure}

The same observation is valid for $AB-$ or $A-$edges-paths; namely, there are such edge-paths that correspond to non-orientable surfaces. The next Lemma rules out edge-paths corresponding to non-orientable surfaces. In order to state the result, we introduce some notation.

Each reduced fraction $p/q$ in $\mathbb{Q}$ can be identified with $0/1$, $1/0$ or $1/1$ by reducing $p$ and $q$ mod 2. An $A-$type edge in $D_t$ is contained in an edge $\langle p_1/q_1, p_2/q_2 \rangle$. If $\{p_1/q_1, p_2/q_2\}$ is identified with $\{0/1, 1/0\}$ mod 2, we say that such and edge is of type $A_0$. On the other hand if $\{p_1/q_1, p_2/q_2\}$ is identified with $\{1/1, 1/0\}$ the edge is said to be of type $A_1$.  

By an $A_iX-$edge-path we will mean an $AX-$edge-path in $D_t$ that consists only of edges of type $X$ and $A_i$ with $i=0,1$ and $X=B, D$. Similarly we use the notation $A_i-$edge-path for an edge-path in $D_1$ that contains only $A_i-$type edges with $i=0,1$. 

\begin{lemma}
Let $S$ be an orientable surface and $\gamma$ be an edge-path in $D_t$ associated to $S$. Suppose that $\gamma$ is an $AX-$edge-path with $X=B,D$. Then ${\gamma}$ is an $A_iX-$-edge-path with $i=0,1$. The same result is valid for $A-$edge-paths. 
\end{lemma}
\begin{proof}
Assume that $\gamma$ contains edges of type $A_0$ and $A_1$. We are going to find a contradiction.

\emph{Case 1: $\gamma$ is an $A$-edge-path in $D_1$.} As $\gamma$ is made of only $A$ type saddles, there must be two consecutive saddles of type $A_0$ and $A_1$. Without loss of generality, we can assume that $A_1$ follows $A_0$. We draw the sequence of pictures module 2 for these two saddles in Fig. \ref{fig:lemma-noA1A0-caseA}. Notice that this is impossible due to orientability of $S$.

\begin{figure}
    \centering
\begingroup%
  \makeatletter%
  \providecommand\color[2][]{%
    \errmessage{(Inkscape) Color is used for the text in Inkscape, but the package 'color.sty' is not loaded}%
    \renewcommand\color[2][]{}%
  }%
  \providecommand\transparent[1]{%
    \errmessage{(Inkscape) Transparency is used (non-zero) for the text in Inkscape, but the package 'transparent.sty' is not loaded}%
    \renewcommand\transparent[1]{}%
  }%
  \providecommand\rotatebox[2]{#2}%
  \newcommand*\fsize{\dimexpr\f@size pt\relax}%
  \newcommand*\lineheight[1]{\fontsize{\fsize}{#1\fsize}\selectfont}%
  \ifx\svgwidth\undefined%
    \setlength{\unitlength}{297.63779528bp}%
    \ifx\svgscale\undefined%
      \relax%
    \else%
      \setlength{\unitlength}{\unitlength * \real{\svgscale}}%
    \fi%
  \else%
    \setlength{\unitlength}{\svgwidth}%
  \fi%
  \global\let\svgwidth\undefined%
  \global\let\svgscale\undefined%
  \makeatother%
  \begin{picture}(1,0.28571429)%
    \lineheight{1}%
    \setlength\tabcolsep{0pt}%
    \put(0,0){\includegraphics[width=\unitlength,page=1]{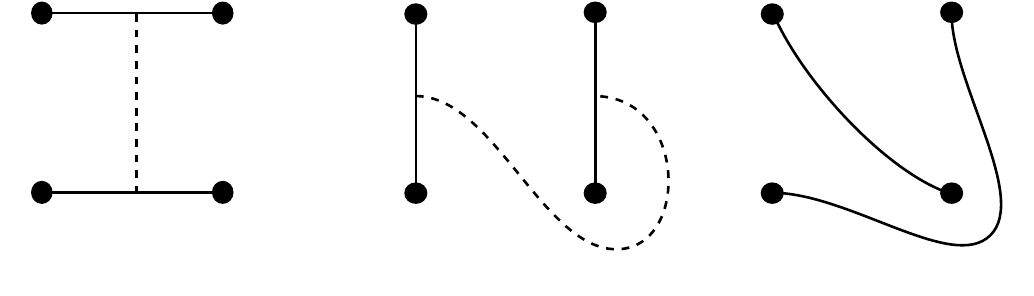}}%
    \put(0.59987544,0.00831824){\color[rgb]{0,0,0}\makebox(0,0)[t]{\lineheight{1.25}\smash{\begin{tabular}[t]{c}$A_1$\end{tabular}}}}%
    \put(0.18878948,0.17886161){\color[rgb]{0,0,0}\makebox(0,0)[t]{\lineheight{1.25}\smash{\begin{tabular}[t]{c}$A_0$\end{tabular}}}}%
  \end{picture}%
\endgroup%

    \caption{No orientability of $A$-edge-path containing both $A_0$ and $A_1$ edges. }
    \label{fig:lemma-noA1A0-caseA}
\end{figure}

\emph{Case 2: $\gamma$ is an $AD$-edge-path in $D_t$.} 
Again, in this case we will have two consecutive saddles of type $A_0$ and $A_1$; because the edge-path comes in blocks of the form $AD\dots A$ where the two $A's$ are of the same type.
The sequence of levels mod 2 is similar to the previous one, see Fig. \ref{fig:lemma-noA1A0-caseAD}, but with some extra $\mu-1$ parallel arcs.

By Lemma \ref{lemma:sillas}a those $\mu -1$ parallel arcs must have all the same orientation; moreover, around the vertices in $K_1 \cap S_\epsilon$, all the arcs are oriented in the same direction. Is not hard to see from Fig. \ref{fig:lemma-noA1A0-caseAD} that it is impossible to give a coherent orientation to all the arcs with the condition that all the $\mu$ parallel arcs have the same orientation, contradicting the orientability of $S$.

\begin{figure}
    \centering
    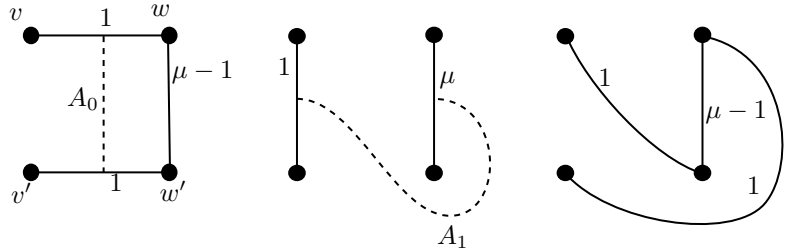
    \caption{No orientability of $AD$-edge-path containing both $A_0$ and $A_1$ edges. }
    \label{fig:lemma-noA1A0-caseAD}
\end{figure}

\emph{Case 3: $\gamma$ is an $AB$-edge-path in $D_t$.} A similar phenomenon to the previous case happens here.  In fact, we get the same picture as in Fig. \ref{fig:lemma-noA1A0-caseAD}. The reason is that the $AB-$edge-paths come in blocks of the form $ABBA$  where the two $A$'s are of the same type. So, if we have two $A$'s of a different type on $\gamma$, there must be two consecutive blocks with different $A$-types.

As consequence of lemma \ref{lemma:sillas}b, all the $\mu-1$ arcs in the first and last level in Fig. \ref{fig:lemma-noA1A0-caseAD} need to be cancel in pairs. And again it becomes impossible to give a coherent orientation satisfying these conditions.

\end{proof}

\begin{remark}
\label{abes}
When an $AB-$edge-path happens it must be of the form $ABBA ABBA \dots ABBA$, where the $A-$ and $B-$type edges lie in different polygons, see Fig. \ref{fig:Dt}. Since the surfaces considered in this work are connected, an $AB-$edge-path consists of at least two $ABBA$ blocks.  
\end{remark}

\begin{remark}
\label{Dcero}
In the case that $S\subset S^3-L_{\beta/\alpha}$ has meridional boundary components in $K_1$ and one boundary component parallel in $K_2$, then the edge-path corresponding to the branched surface that carries $S$ belong to the diagram $D_0$. Thus it is an $BD-$edge-path. For $B-$type edges to exist and to obtain an orientable surface it must happen that $\rho$ greater than $1$.
See Figures \ref{fig:b-d0} and \ref{fig:d-d0} of $B-$ and $D-$ type saddle for $t=0$. We conclude that in this case, the edge-path consists only of $D-$type edges.
\end{remark}

\begin{figure}[ht]
\centering
    \begin{subfigure}[b]{0.5\textwidth}
        \includegraphics[width=\textwidth]{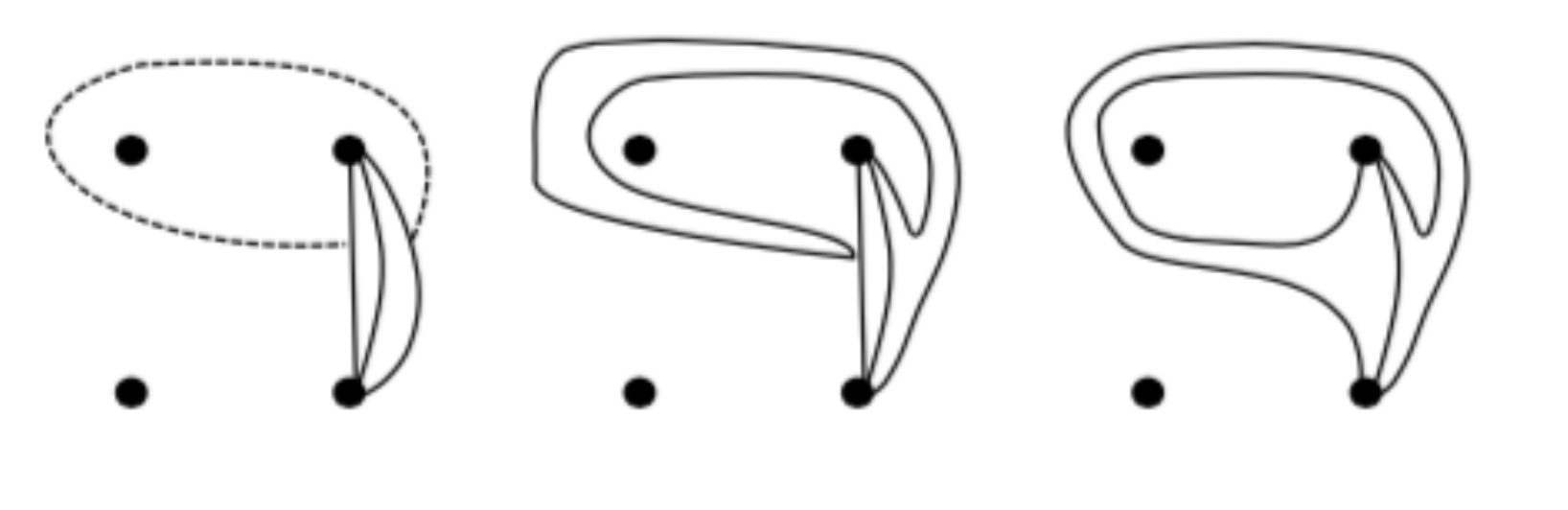}
        \caption{B-type saddle in $D_0$}
        \label{fig:b-d0}
    \end{subfigure}
    \vfill
   
    \begin{subfigure}[b]{0.2\textwidth}
        \includegraphics[width=\textwidth]{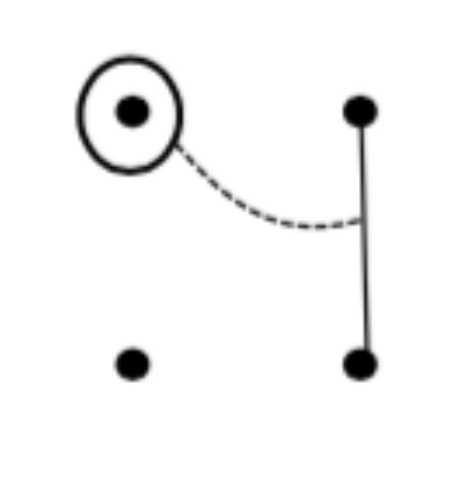}
        \caption{D-type saddle in $D_0$}
        \label{fig:d-d0}
    \end{subfigure}
 \end{figure}

\subsection{Continued fractions and genus of surfaces}
\label{fractions and genus}
\label{continuedfrac}

 We recall from Hatcher and Thurston \cite{HT}:
 an edge-path from $1/0$ to $\beta/\alpha$ in the diagram $D_1$ corresponds uniquely to a continued fraction expansion $\beta/\alpha=[r; b_1,..., b_k]$, where the partial sums $\beta_i/\alpha_i=[r; b_1,..., b_i]$ are the successive vertices of the edge-path. 
 \begin{equation*}
     \frac{\beta}{\alpha}= r+ \cfrac{1}{b_1+\cfrac{1}{b_2+\cfrac{1}{\ddots + \cfrac{1}{b_k}}}}
 \end{equation*}

 \begin{remark}
 \label{girosizqder}
 At the vertex $\beta_i/\alpha_i$ the path turns left or right across $\vert b_i \vert$ triangles.
 For $i-$odd, right if $b_i>0$ and left if $b_i<0$. For $i-$even left if $b_i>0$ and right if $b_i<0$. The number of $C$ diagonals is $\vert b_i \vert/2$
 \end{remark}

 By Remark \ref{Dtcolapsado}, in $D_t$ the diagonals $C$ of the diagram $D_1$ are changed by inscribed rectangles. So for each diagonal $C$ we obtain a $D-$edge around the vertex $\beta_i/\alpha_i$, see Fig \ref{fig:D1}. Thus the number of $D-$edges around $\beta_i/\alpha_i$ is $\vert b_i \vert /2$.

 In this paper we use two special types of continued fraction expansions: 
$\beta/\alpha=[0;2n_1,2n_2,\dots,2n_j]$ and $\beta/\alpha=[1;2m_1,2m_2,\dots,2m_i]$. These are the unique continued fraction where each entry is an even number and $j,i$ are odd.

We will described the edge-path in $D_t$ associated to these continued fractions, such that the branched surface associated carries a connected, compact, essential and orientable surface $S\subset S^3- L_{\beta/\alpha}$ with one boundary component parallel to $K_2$ and $n$-boundary components in $K_1$, which are non-meridional and $n\neq 0$; i.e, $\rho=1$ and $\mu$ is a multiple of $n$. For now on we assume that $t\neq 0, \infty$.

For short we will say that the surface $S$ is associated to the edge-path. We will compute the genus of $S$ as well.

For both continued fraction expansions, the vertices $\beta_i/\alpha_i$, given by the partial sums, satisfy that $\alpha_{2k+1}$ is even and $\alpha_{2k}$ is odd. 

In the diagram $D_1$, the edge-path for $[0;2n_1,2n_2,\dots,2n_j]$ passes by $0/1$; and the edge-path for $[1;2m_1,2m_2,\dots,2m_i]$  passes by $1/1$. These are $A-$edge-paths.

The edge-path corresponding to the continued fraction $[0;2n_1,2n_2,\dots,2n_j]$ is an $A_0-$edge-path, and the corresponding to the continued fraction $[1;2m_1,2m_2,\dots,2m_i]$ is a  $A_1-$edge-path.

If $\mu=1$, then the edge-path just obtained; is the one that corresponds to $S$. Hence we obtain an edge-path of length $j+1$ ( or $i+1$), where each edge lies in different triangles by construction. For each $A-$type edge we have an $A-$type saddle, thus we can compute the genus of $S$ using Euler characteristic.

\begin{proposition}
\label{genusA}
Let $[r; 2r_1,..., 2r_k]$ be one of two the continued fractions expansions for $\beta/ \alpha$. If $\mu=1$, the associated $A-$edge-path consisting of $k+1$ edges; corresponds to a connected, compact, essential and orientable surface $S\subset S^3- L_{\beta/\alpha}$ with one boundary component parallel to $K_i$ for $i=1,2$. Then the genus of $S$ is 
\begin{equation*}
\frac{1}{2}(k-1)
\end{equation*}
\end{proposition}
\qed

If $\mu \neq 1$, we pass to the $D_t$ diagram with $t \neq 1$. Each edge $A$ in $D_1$ is changed into an  $A-$edge and a $B-$edge. The edge path in $D_1$ is transformed into an $AB$-edge-path in a diagram $D_t$. Around a vertex with even denominator there are only $A-$type edges, and around a vertex with odd denominator there are only $B-$type edges. Thus the pattern $ABBA$ is repeated $\frac{1}{2}(i+1)$-times (or $\frac{1}{2}(j+1)$-times). 

Observe that an $AB-$edge-path obtained as above may not correspond to a minimal edge-path in $D_t$; nevertheless a minimal $AB-$edge-path associated to a connected, compact, essential and orientable surface is in correspondence with an $A_i-$edge-path with $i=0,1$. A condition on the continued fraction expansion $[r; 2r_1,..., 2r_k]$ for $\beta/ \alpha$  for an $AB-$edge-path to be minimal is that $\vert r_j \vert$ greater than 1 for all $j$. 

If an orientable surface $S$ is carried by this kind of path, Lemma \ref{lemma:sillas} implies $\mu=n$ and by Remark \ref{abes} we have $\frac{1}{2}(i+1)\geq 2$ (or $\frac{1}{2}(j+1)\geq 2$; since we require a connected surface, where $i,j$ are the lengths of the continued fraction expansions for $\beta/\alpha$.
Hence an $AB$-edge-path that passes trough the vertices $0/1$ or $1/1$ associated to an orientable surfaces must contained at least two blocks of the pattern $ABBA$, thus the continued fraction expansion must contain at least three even terms, after the 0 or 1 entries. 

In order to compute the genus of $S$, the associated surface to this edge-path, we count the number of saddles corresponding to the edge-path. Observe that each $A-$type edge corresponds to one saddle and each $B-$type edge to $\frac{1}{2}(n-1)$-saddles. Each block of $ABBA$ contributes with $(n+1)-$saddles. Again, using Euler characteristic we find:

\begin{proposition}
\label{genusAB}
Let $[r; 2r_1,..., 2r_k]$ be one of two the continued fractions expansions for $\beta/ \alpha$, with $k\geq 3$ and $\vert r_t \vert \geq 2$ for all $t$. If $\mu=n$ and the associated $AB$-edge-path; consisting of $\frac{1}{2}(k-1)$ $ABBA$ blocks; corresponds to a connected, compact, essential and orientable surface $S\subset S^3- L_{\beta/\alpha}$ with one boundary component parallel to $K_2$ and $n$-boundary components parallel to $K_1$. Then the genus of $S$ is:
\begin{equation*}
 1+\frac{(n+1)(k-3)}{4}
\end{equation*}
\end{proposition}
\qed

If $S$ is oriented and $\mu\neq n$ then the edge-path for $S$ is an $AD$-edge-path.
In this case, we substitute each pair $BB$ in the above edge-path by a sequence $DD...D$, where the number of $D$'s is given by the number of diagonals $C$ in the diagram $D_1$ around the corresponding vertex. For instance, if $\beta_{2k}/\alpha_{2k}=[0;2n_1,2n_2,\dots,2n_{2k}]$, the number of $D$'s is $\vert n_{2k+1}\vert$.

Summarizing, the $AD-$edge-type in $D_t$ associated to the continued fraction expansion  $[0;2n_1,2n_2,\dots,2n_j]$ is \\
$A\underbrace{DD...D}_{\vert n_1\vert}AA\underbrace{DD...D}_{\vert n_3\vert}AA...AA\underbrace{DD...D}_{\vert n_{j}\vert}A$. Notice that the two consecutive $A-$type edges belong to different triangles, and the $D-$type edges belong to different quadrilaterals by construction. Thus we obtain a minimal edge-path. 
Analogously, for the continued fraction expansion $[1;2m_1,2m_2,\dots,2m_i]$ we associate an $AD-$edge-path.

Next we compute the genus of such $S$.

\begin{proposition}
\label{genusAD}
Let $[r; 2r_1,..., 2r_k]$ be one of two the continued fractions expansions for $\beta/ \alpha$. If $\mu \neq n$ and the associated path 
$A\underbrace{DD...D}_{\vert r_1\vert}AA\underbrace{DD...D}_{\vert r_3\vert}AA...AA\underbrace{DD...D}_{\vert r_{k}\vert}A$ corresponds to a connected, compact, essential and orientable surface $S\subset S^3- L_{\beta/\alpha}$ with one boundary component parallel to $K_2$ and $n$-boundary components non-meridional on $K_1$. Then the genus of $S$ is:
\begin{equation*} 
\frac{1}{2} \Bigg[\displaystyle{\left(-1 + \sum_{ h \; odd}|r_h|\right) } \; (|\mu| - 1) + (k + 1) - (n+1)\Bigg]
\end{equation*} 
where $h\in \{1, ..., k\}$
\end{proposition}
\begin{proof}

Use Euler characteristic, considering that each $A-$type edge corresponds to one saddle and each $D-$type edge corresponds to $\mu-1$ saddles.
\end{proof}

\subsection{Boundary slopes}
The boundary of a branched surface derived from the Hatcher-Floyd construction defines a train track on the boundary of the regular neighborhood of the link. Thus the boundary of any essential surface $S$ carried by the branched surface is carried by this train track. Lash, \cite{L}, calculated the space of boundary slopes for the Whitehead link.

In the following paragraph we explain Lash algorithm. We base the explanation on the article \cite{GHS}:

To compute the boundary slopes of the surfaces the frame used consists of the meridian $\mu_i$ and a non-standard longitude $\lambda_i$ of $K_i$. In $S^2=\mathbb{R}^2/\Gamma$, we take the arc $s$ of slope $0$ connecting $\Gamma(0,0)$ and $\Gamma(0,1)$. $\lambda_1$ is the union of the arc $(s \times [0,1]) \cap \bd N(K_1)$ and an arc in $(S^2 \times [1, \infty)) \cap \partial N(K_1)$. $\lambda_1$ is oriented toward increasing $r \in [0,1]$ along the axis $\Gamma(0,0)\times [0,1]$. The meridian $\mu_1$ is oriented as a right-handed circle around the axis $\Gamma(0,0)\times [0,1]$ oriented upward. We obtain $\lambda_2, \mu_2$ from $\lambda_1, \mu_1$ by \(\)rotating by $180^\circ$ about the axis $\Gamma(1/2, 1/2) \times [0,1]$.

Let $i_j$ be the algebraic intersection number $\partial S\cdot \lambda_j$ in $\partial N(K_j)$.
Let $\varphi$ be the map such that for $s\in [0,1]$, $\varphi(s)=(i+s-1)/k \in [(i-1)/k,i/k]$. For $0\leq t < 1$, $\partial \Sigma_{E_i}=\partial (g\times \varphi)(\Sigma_{e_0})$, $g= \left( \begin{array}{cc} a & b \\ c & d \end{array} \right)\in G$ contributes to the number $i_j$ as in Table \ref{tabla1}, if the orientations of $e_i$ and $g(e_0)$ agree. If they disagree, we change all the sign of the number in Table \ref{tabla1}.

We calculate the boundary slope of a surface $S$ corresponding to an $AD$-edge-path. 

Taking the sum of the entries of the row of $i_1$ and $i_2$ of Table \ref{tabla1}, we can see that the slope on $\partial N(K_1)$ is $(\mu, r(\mu-\rho)+s\rho)$ and that on $\partial N(K_2)$ is $(\rho, v(\mu-\rho)+s\rho)$, where the first coordinate is the longitudinal entry, and the second coordinate is the meridional entry with respect to the unusual longitude $\lambda_i$. To obtain the real slope, we need to know the slope of the preferred longitude, which is obtained by substituting $1$ for $\mu$ and $0$ for $\rho$, in $\partial N(K_1)$, which is $(1, r)$. The preferred longitude of $K_2$ is of slope $(1, s-v)$, which is obtained by substituting $\rho=1$ and $\mu= 0$ in $\partial N(K_2)$. But the preferred longitude of $K_2$ is the same for $K_1$, recall that we take $\lambda_2$ as the image of $\lambda_1$ by rotating $180^\circ$ about the axis $\Gamma(1/2, 1/2) \times [0,1]$.  Thus, $(1, r)= (1, s-v)$ and $s-r=v$. The slopes with respect to the preferred longitude can be obtained from  $(\mu, r(\mu-\rho)+s\rho- r\mu)=(\mu,(s-r)\rho)=(\mu, v\rho)$ on $\partial N(K_1)$, and  $(\rho, v(\mu-\rho)+s\rho-(s-v)\rho)=(\rho, v\mu)$ on $\partial N(K_2)$.

\begin{table}[ht]
\begin{tabular}{ |c|c|c|c| }
\hline
Label & condition on $-d/c$ & $i_1$ & $i_2$\\ \hline
\multirow{3}{*}{A} & $-\infty < -\frac{d}{c}< 0$ & $\rho$ & $\rho$ \\ \cline{2-4}
 & $0< -\frac{d}{c}< \infty$ & $-\rho$ & $-\rho$ \\ \cline{2-4}
 & $-\frac{d}{c}=0, \pm \infty$ & $0$ & $0$ \\ \hline
 
 \multirow{3}{*}{B} & $-\infty < -\frac{d}{c}< 0$ & $-(\mu-\rho)$ & 0  \\ \cline{2-4}
 & $0< -\frac{d}{c}< \infty$ & $\mu-\rho$ & 0 \\ \cline{2-4}
 & $-\frac{d}{c}=0, \pm \infty$ & $0$ & $0$ \\ \hline
 
 \multirow{3}{*}{C} & $ 0 < -\frac{d}{c}< 1$ & $-2\rho$ & 0  \\ \cline{2-4}
 & $ -\frac{d}{c}=0,1$ & $-\rho$ & $\rho$ \\ \cline{2-4}
 & otherwise & $0$ & $2\rho$ \\ \hline

\multirow{3}{*}{D} & $\frac{1}{2}< -\frac{d}{c}< \infty$ & $\mu-\rho$ & $\mu-\rho$ \\ \cline{2-4}
 & $-\frac{d}{c}=\frac{1}{2}, \pm \infty$  & $0$ & $\mu-\rho$ \\ \cline{2-4}
 & otherwise & $-(\mu-\rho)$ & $\mu-\rho$ \\ \hline
 \end{tabular}
\caption{\label{tabla1}}
\end{table}

Recall that we are considering two types of continued fraction expansions for $\beta/\alpha$, namely $F_0=[0;2n_1,2n_2,\dots,2n_j]$ and $F_1=[1;2m_1,2m_2,\dots,2m_i]$. As discussed in Section \ref{continuedfrac}, for each continued fraction there is an $AD$-edge-path corresponding to an essential surface. We will determined the contribution of $v$. 

The edge path for $F_0$ is $A\underbrace{DD...D}_{\vert n_1\vert}AA\underbrace{DD...D}_{\vert n_3\vert }AA...AA\underbrace{DD...D}_{\vert n_{j}\vert}A$, the orientations of the edges $A$ and $D$ need to be determined in order to compute $v$. If the orientation of $e_i \in \{ A, D\}$ and $g(e_0)$ with $e_0\in \{ A_0, D_0\}$ agree we denote the edge by $\overrightarrow{e_i}$, if they disagree we denote it by $\overleftarrow{e_i}$.

By the construction of the edge path is not hard to see that, see Figure \ref{fig:ADpath1}:
\begin{enumerate}
\item The first $A-$type edge is an $\overrightarrow{A}$. 
\item The first $\vert n_1\vert$ $D-$type edges are $\overleftarrow{D}$. 
\item Each intermediate pair $AA$ is of the form $\overleftarrow{A}\ \overrightarrow{A}$.
\item The last $A-$type edge is $\overleftarrow{A}$
\end{enumerate}

\begin{figure}[ht]

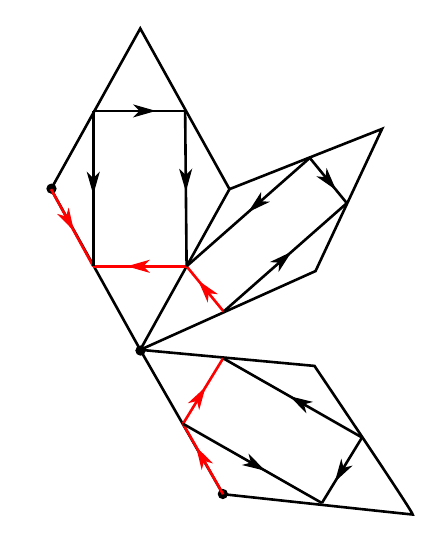
\caption{Path $\ora{A}\ola{D}\ola{D}\dots\ola{D}\protect\ola{A}$ }
\label{fig:ADpath1}
\end{figure}

For the remaining $D-$type edges we have:

\begin{proposition} 
For the continued fraction expansion $F_0$ and $i$ odd.
\begin{enumerate}
\item If $n_i > 0$ then the sequence of $D-$type edges are $\overleftarrow{D}$ .
\item If $n_i < 0$ then the sequence of $D-$type edges are $\overrightarrow{D}$.
\end{enumerate}
\end{proposition}
\begin{proof}
For both cases we need to verify that the agreement or disagreement of the $D-$types edges with $g(D_0)$ at the $i-$th position for $i$ odd. Since we are considering the continued fraction $F_0$, all the vertices $\beta_i/ \alpha_i$, for $i$-odd, are congruent with $0/1$ mod 2, up to transformations of elements of $PSL(2,\mathbb{Z})$. Thus the $D-$type edge at such vertex $\beta_i/\alpha_i$ is a $\overleftarrow{D}$ edge. See Figure \ref{fig:ADpath1}.  From Remark \ref{girosizqder} the quadrilateral turns right if $n_i>0$ and left if $n_i<0$. Hence, if $n_i>0$ the sequence of $D-$type edges are $\overleftarrow{D}$ and if  $n_i < 0$  the sequence of $D-$type edges are $\overrightarrow{D}$. See Figures \ref{fig:ADpath1} and \ref{fig:giroizq} for the turns around $\beta_i/ \alpha_i$ mod 2.
\end{proof}

\begin{figure}[ht]
\includegraphics[width=5cm]{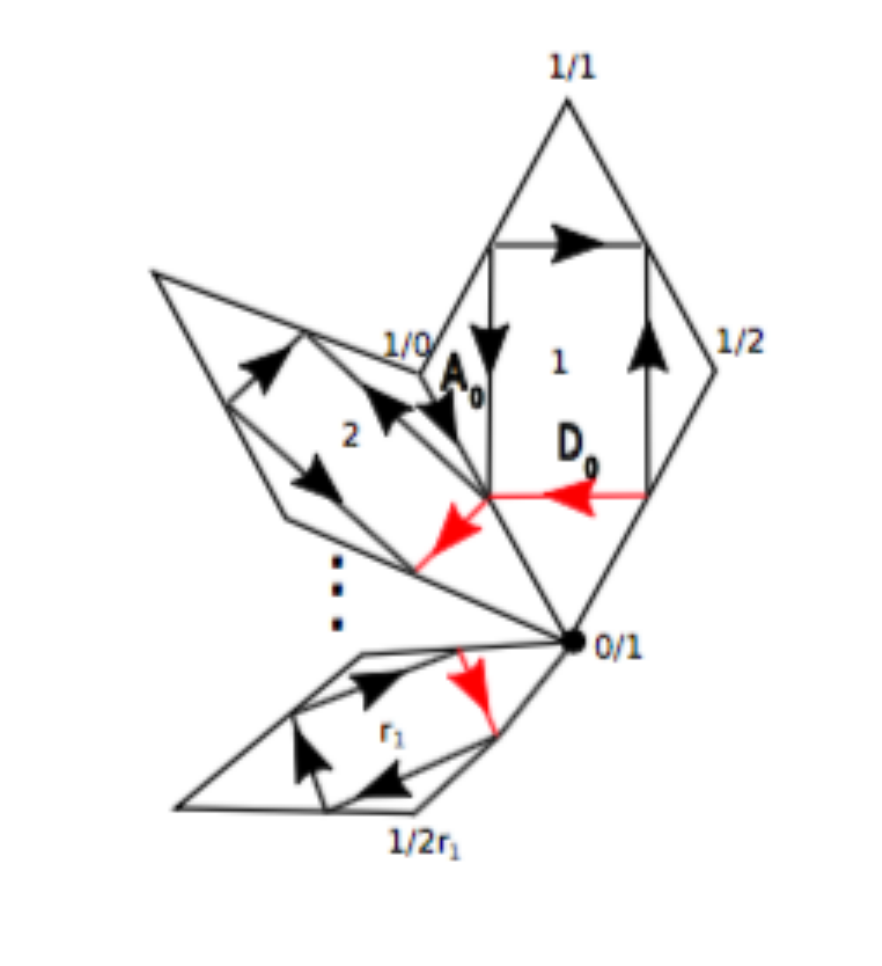}
\caption{Path $\ora{D}\ora{D}\cdots\ora{D}$ }
\label{fig:giroizq}
\end{figure}

The value of $v$ for the edge path corresponding to the continued fraction expansion $F_0$ is $v=-(n_1 + n_3 +...+n_j)$ because when $n_i > 0$ we see $\overleftarrow{D}-$type edges, so the contribution in the Table \ref{tabla1} is $-n_i$, and if $n_i < 0$ we see $\overrightarrow{D}-$type edges, so they contribute with $-n_i$ in Table \ref{tabla1}. Since $lk(K_1, K_2)= n_1 + n_3 +...+n_j$, we conclude the following:

\begin{corollary}
\label{linkingnegativo}
Let $S$ be a surface associated to the edge path \\ $A\underbrace{DD...D}_{\vert n_1\vert}AA\underbrace{DD...D}_{\vert n_3\vert }AA...AA\underbrace{DD...D}_{\vert n_{j}\vert}A$, arising from  $[0;2n_1,2n_2,\dots,2n_j]$. 
The boundary slopes of $S$ with respect to the preferred longitude on $\partial N(K_1)$ is $(\mu, -lk(K_1, K_2)\rho)$ and on $\partial N(K_2)$ is $(\rho, -lk(K_1, K_2)\mu)$.

\end{corollary}

On the other hand; for the continued fraction expansion $F_1=[1;2m_1,2m_2,\dots,2m_i]$ the corresponding edge path in the diagram $D_t$ is $A\underbrace{DD...D}_{\vert m_1\vert}AA\underbrace{DD...D}_{\vert m_3\vert }AA...AA\underbrace{DD...D}_{\vert m_{i}\vert}A$. This path lies in the same sequence of quadrilaterals as the corresponding path for the continued fraction expansion $F_0$, but it is made of the $A-$ and $D-$type edges which do not belong to the path for $F_0$.
Reasoning as before, we have that for the $AD-$edge-path corresponding to $F_1$:

\begin{enumerate}
\item The first $A-$type edge is an $\overrightarrow{A}$. 
\item The first $\vert n_1\vert$ $D-$type edges are $\overrightarrow{D}$. 
\item Each intermediate pair $AA$ is of the form $\overleftarrow{A}\ \overrightarrow{A}$.
\item The last $A-$type edge is $\overleftarrow{A}$
\end{enumerate}

\begin{proposition} 
For the continued fraction expansion $F_1$ and $i$ odd.
\begin{enumerate}
\item If $n_i > 0$ then the sequence of $D-$type edges are $\overrightarrow{D}$ .
\item If $n_i < 0$ then the sequence of $D-$type edges are $\overleftarrow{D}$.
\end{enumerate}
\end{proposition}

\begin{corollary}
\label{linkingpositivo}
Let $S$ be a surface associated to the edge path \\ $A\underbrace{DD...D}_{\vert m_1\vert}AA\underbrace{DD...D}_{\vert m_3\vert }AA...AA\underbrace{DD...D}_{\vert m_{i}\vert}A$, arising from  $[1;2m_1,2m_2,\dots,2m_i]$. 
The boundary slopes of $S$ with respect to the preferred longitude on $\partial N(K_1)$ is $(\mu, lk(K_1, K_2)\rho)$ and on $\partial N(K_2)$ is $(\rho, lk(K_1, K_2)\mu)$.

\end{corollary}

Analogously, we can compute the boundary slopes for $A$-edge-path and $AB$-edge-path, in both cases the resulting boundary slopes are equal to zero.

\section{Fiberings}
\label{fiberings}
Floyd and Hatcher give a criterion to determine when a surface $S$ in $S^3-L_{\beta/\alpha}$ is a fiber of a fibering $S^3-L_{\beta/\alpha} \rightarrow S^1$.

\begin{definition}
Let $\gamma$ be a path in $D_t$, with $t\in[0, \infty]$. A maximal sequence of consecutive $A-$ and $D-$type edges in $\gamma$ each separated from the next by only one edge in $D_t$, is called a string.

\end{definition}
Figure \ref{string} shows an example of a string and Figure \ref{nstring} a path which is not a string.

\begin{figure}
    \centering
    \begin{subfigure}[b]{0.3\textwidth}
        \includegraphics[width=\textwidth]{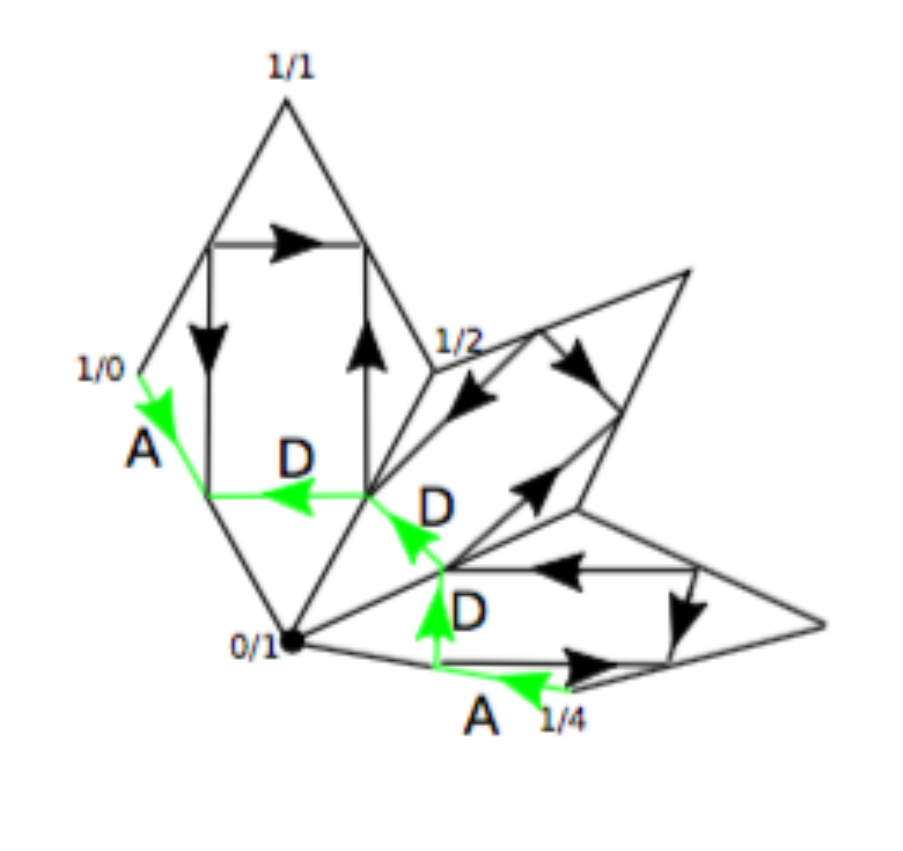}
        \caption{a string}
        \label{string}
    \end{subfigure}
    
    \begin{subfigure}[b]{0.4\textwidth}
        \includegraphics[width=\textwidth]{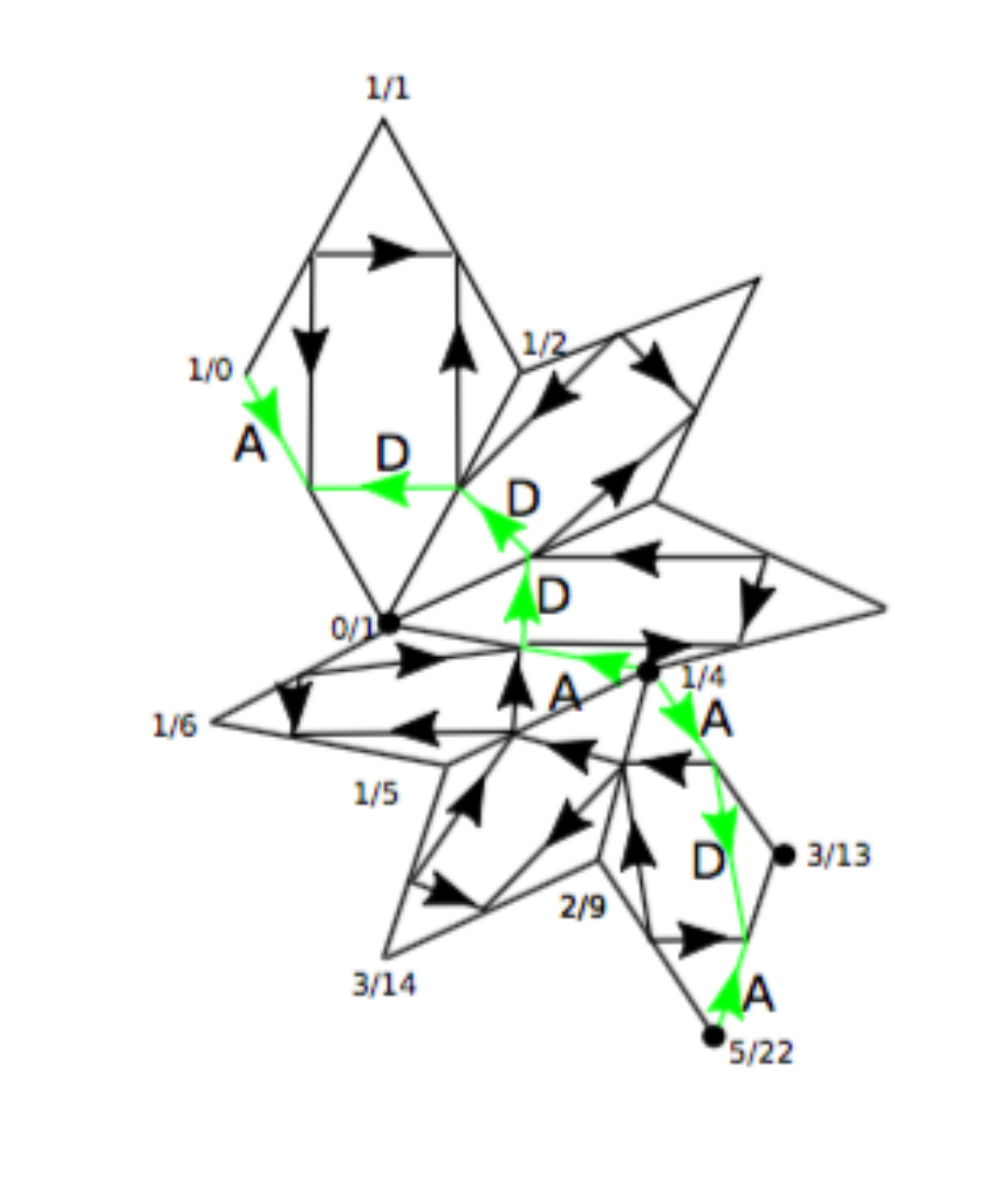}
        \caption{not a string}
         \label{nstring}
   \end{subfigure}
   \caption{}
\end{figure}

Proposition 6.1 of \cite{FH} states sufficient and necessary conditions for fibering:

\begin{proposition}
A surface in $S^3-L_{\beta/\alpha}$ is a fiber of a fibering $S^3-L_{\beta/\alpha} \rightarrow S^1$ if and only if it is isotopic to a surface carried by a branched surface $\Sigma_{\gamma}$ whose associated edge-path $\gamma$ from $1/0$ to $\beta/\alpha$, in a determined $D_t$, consists of a single string of $A-$ and $D-$ type edges. 
\end{proposition}

The following theorem tells us conditions on the continued fraction expansion, considered in this work, for a surface $S$ to correspond to a fiber of a fibering $S^3-L_{\beta/\alpha} \rightarrow S^1$.

\begin{theorem}
\label{fiberingtheorem}
\label{fibering}
Let $L_{\beta/\alpha}$ be a link and $S$ a surface in $S^3-L_{\beta/\alpha}$. 

\begin{enumerate}
    \item Suppose $S$ is associated to an $AD$-edge-path. $S$ is a fiber of a fibering $S^3-L_{\beta/\alpha} \rightarrow S^1$ if and only if the continued fraction expansion for $\beta/\alpha$ has the form $[r; 2r_1, 2\epsilon_2, 2r_3, ... ,2\epsilon_{n-1}, 2r_n ]$ with $r=0,1$ and $|\epsilon_i|=1$.
    
    \item Suppose $S$ is associated to an $A$-edge-path. $S$ is a fiber of a fibering $S^3-L_{\beta/\alpha} \rightarrow S^1$ if and only if the continued fraction expansion for $\beta/\alpha$ has the form $[r; 2\epsilon_1, 2\epsilon_2, ..., 2\epsilon_n]$ with $r=0,1$ and $|\epsilon_i| =1$ for all $i$.
    
    \item Suppose $S$ is associated to a $D$-edge-path. $S$ is a fiber of a fibering $S^3-L_{\beta/\alpha} \rightarrow S^1$ if and only if the continued fraction expansion for $\beta/\alpha$ has the form $[0; 2r_1,-2, 2r_2, ... -2, 2r_n ]$ with $2r_i$ positive for all $i$. Thus the fraction starting with 1, is of the form $[1; 2n_1, 2, 2n_2, ... , 2, 2n_j]$ with $2n_k$ negative for all $k$.
    
\end{enumerate}
\end{theorem}
\begin{proof}
In each case we need to verify that the corresponding path in the adequate diagram $D_t$ is a string.
\begin{enumerate}
    \item Let $\gamma= A\underbrace{DD...D}_{\vert r_1\vert}AA\underbrace{DD...D}_{\vert r_3\vert }AA...AA\underbrace{DD...D}_{\vert r_{k}\vert}A$
be the edge-path arising from the continued fraction expansions $[r;2r_1,2r_2,\dots,2r_k]$. Observe that any two consecutive $A$ and $D$ are separated by exactly one $C-$type edge in $D_t$, with $t\neq 0,1, \infty$. And any two consecutive $D-$type edges are separated by exactly one $A-$ or $B-$type edge, see Figures \ref{string} and \ref{nstring}. To guarantee that $\gamma$ is a single string, it is necessary to check when two consecutive $A-$type edges are separated by only one edge. By inspecting Figure \ref{nstring}, is easy to observe that  two $A-$type edges are separated by only one edge if $2r_2=2, -2$. This pattern is extended to the whole path $\gamma$. Thus the condition is that $r_l= 2{\epsilon}$ with $\epsilon=-1,1$ for $l$-even in the continued fraction expansion $[r;2r_1,2r_2,\dots,2r_k]$. 

\item Consider the continued fraction expansion $[r; 2r_1,..., 2r_j]$, since $S$ is associated an $A$-edge-path $\gamma$ in $D_1$ diagram, $\gamma$ goes through all the vertices $1/0, \beta_0/\alpha_0, \beta_1/\alpha_1, ..., \beta_j/\alpha_j=\beta/\alpha$. For $\gamma$ to be a string, every two consecutive $A$-type edges must be separated by exactly one $C$-type edge or by exactly on $A$-type edge. There are two possibilities depicted in Figures \ref{pathA-1} and \ref{pathA-2}, we see that $\vert 2r_i \vert=2$ for all $i$. Thus the continued fraction expansion has the form $[r; 2\epsilon_1, 2\epsilon_2, ..., 2\epsilon_n ]$ with $\epsilon_i = \pm 1$.

\item Let us consider the continued fraction expansion $[0; 2r_1,..., 2r_j]$, in this case the surface $S$ is in correspondence with a $D$-edge-path $\gamma$ in the $D_0$ diagram. The first $r_1$ edges of type $D$ pass through vertices $1/0, 1/2, 1/4, ... 1/2r_1= \beta_1/\alpha_1$. Each two consecutive $D$-type edges are separated by exactly one $B$-type edge. Thus, that piece of $\gamma$ satisfies the condition to be a string. See Figure \ref{pathD-1}. A similar phenomenon occurs around a vertex $\beta_i/\alpha_i$ with $i$-even. It is necessary to determine when two consecutive $D$-edges with common vertex $\beta_i/\alpha_i$ with $i$-odd are separated by exactly one $B$-edge.

Next we will determine conditions for $r_2, r_3$ in order to keep $\gamma$ being a string, up to $PSL_2(\mathbb{Z})$ transformation, we will be able to argue that the conditions for $r_2, r_3$ can be extended to the following $r_i's$.

First let us consider $2r_2, 2r_3$ both positive. The $B$-edge connecting $0/1$ and $1/2r_1$ has to turn left $2r_2$-edges to reach the vertex $\beta_2/\alpha_2$. Then the edge connecting $1/2r_1$ and $\beta_2/\alpha_2$ has to turn right $2r_3$-edges to reach vertex $\beta_3/\alpha_3$. Recall that the turns at each vertex was described in Remark \ref{girosizqder}, for the situation just described see Figure \ref{pathD-1}. The two consecutive $D$-edges with common vertex $1/2r_1$ are separated by $2r_2+2r_3-1$ $B$-edges, since $2r_2,2r_3 \geq 2$, there are at least three $B$-edges in between. Hence this situation will not give a string.

Secondly consider $2r_2$ positive and $2r_3$ negative. In this case, 
the $B$-edge connecting $0/1$ and $1/2r_1$ has to turn left $2r_2$-edges to reach the vertex $\beta_2/\alpha_2$. Then the edge connecting $1/2r_1$ and $\beta_2/\alpha_2$ has to turn left $2r_3$-edges to reach vertex $\beta_3/\alpha_3$. The two consecutive $D$-edges with common vertex $1/2r_1$ are separated by $2r_2$ $B$-edges, since $2r_2 \geq 2$, there are at least two $B$-edges in between. Hence this situation will not give a string. See Figure \ref{pathD-2}

Thirdly suppose $2r_2$ and $2r_3$ are negative. The $B$-edge connecting $0/1$ and $1/2r_1$ has to turn right $2r_2$-edges to reach the vertex $\beta_2/\alpha_2$. Then the edge connecting $1/2r_1$ and $\beta_2/\alpha_2$ has to turn left $2r_3$-edges to reach vertex $\beta_3/\alpha_3$.  See Figure \ref{pathD-3}. The two consecutive $D$-edges with common vertex $1/2r_1$ are separated by $\vert 2r_2\vert +\vert 2r_3\vert-1$ $B$-edges, since $2r_2,2r_3 \geq -2$, there are at least three $B$-edges in between. Thus this case will not give a string.

Finally, if $2r_2$ is negative and $2r_3$ is positive. The $B$-edge connecting $0/1$ and $1/2r_1$ has to turn right $2r_2$-edges to reach the vertex $\beta_2/\alpha_2$. Then the edge connecting $1/2r_1$ and $\beta_2/\alpha_2$ has to turn right $2r_3$-edges to reach vertex $\beta_3/\alpha_3$.  See figure \ref{pathD-4}. In this case the edges with common vertex $1/2r_1$ are separated by $\vert 2r_2\vert -1$ $B$-edges, so to obtain a string is necessary to $2r_2=-2$.

At this point we have that the continued fraction expansion looks like $[0; 2r_1, -2, 2r_3, x_4, ...,x_n]$.

Using a transformation in $PSL_2(\mathbb{Z})$, we can put in correspondence $\beta_1/ \alpha_1 \rightarrow 1/0$, $\beta_2/\alpha_2 \rightarrow 0/1$ and $\beta_3/\alpha_3 \rightarrow \beta_1/\alpha_1$. Analysing as above we are able to conclude that $2r_4=-2$ and $2r_5$
is positive. Thus, if we keep doing the correspondence for the remaining vertices, we conclude that the continued fraction expansion has the form $[0; 2r_1, -2, 2r_3, -2, ..., -2, 2r_n]$ with $2r_i$ positive for all $i$-odd. A similar analysis shows that the other continued fraction expansion must be $[1; 2n_1, 2, 2n_2, ... , 2, 2n_j]$ with $2n_k$ negative for all $k$.

\end{enumerate}
\end{proof}
\begin{figure}
    \centering
    \begin{subfigure}[b]{0.3\textwidth}
        \includegraphics[width=\textwidth]{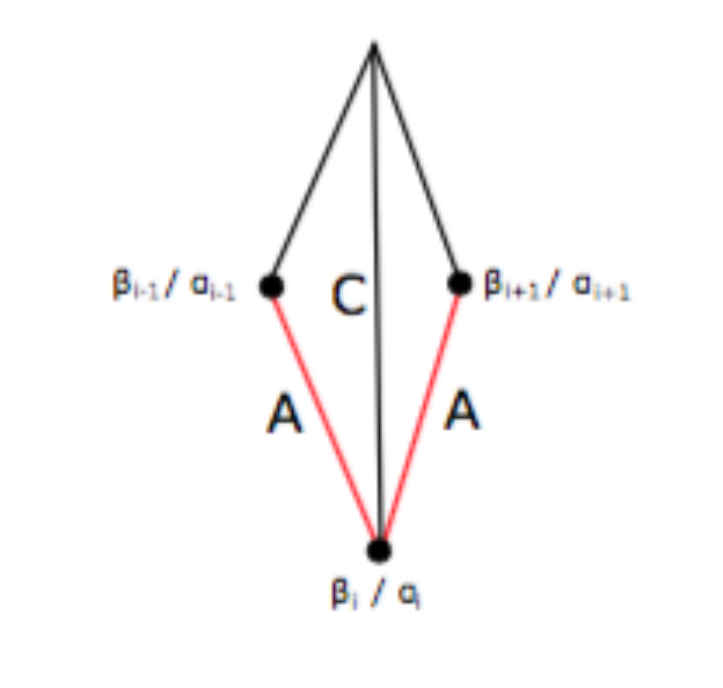}
        \caption{$2r_2$ and $2r_3$ positive}
        \label{pathA-1}
    \end{subfigure}
      \begin{subfigure}[b]{0.4\textwidth}
        \includegraphics[width=\textwidth]{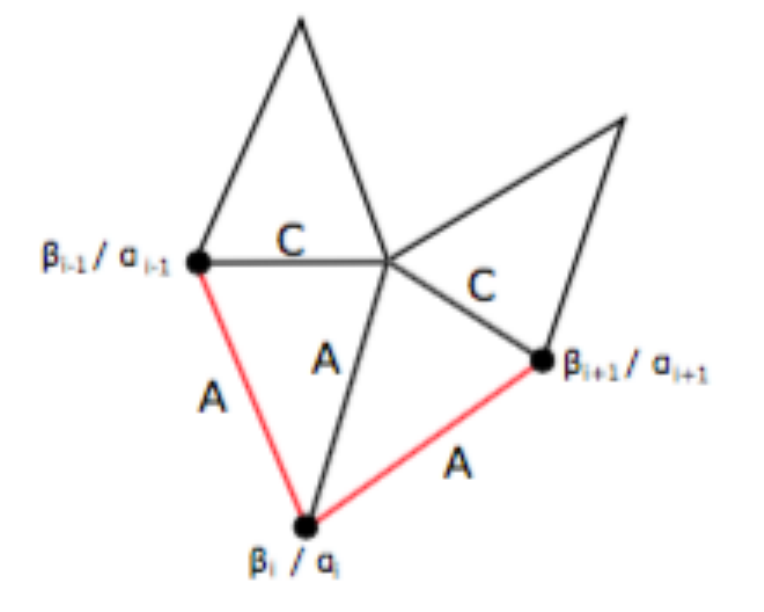}
        \caption{$2r_2$ positive and $2r_3$ negative}
         \label{pathA-2}
   \end{subfigure}
   
   \caption{Possibilities for $A$-edges in $D_1$ to belong to a string.}
\end{figure}

\begin{figure}
    \centering
    \begin{subfigure}[b]{0.3\textwidth}
        \includegraphics[width=\textwidth]{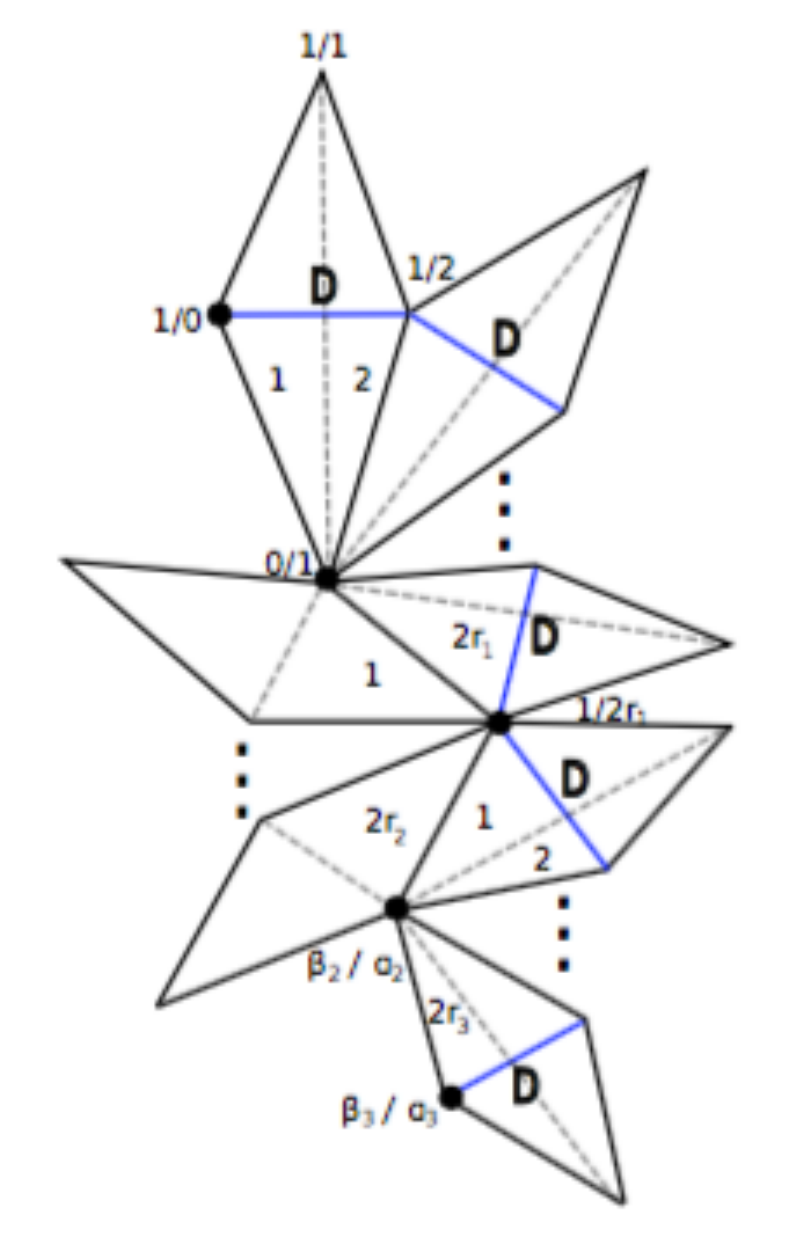}
        \caption{$2r_2$ and $2r_3$ positive}
        \label{pathD-1}
    \end{subfigure}
   
    \begin{subfigure}[b]{0.4\textwidth}
        \includegraphics[width=\textwidth]{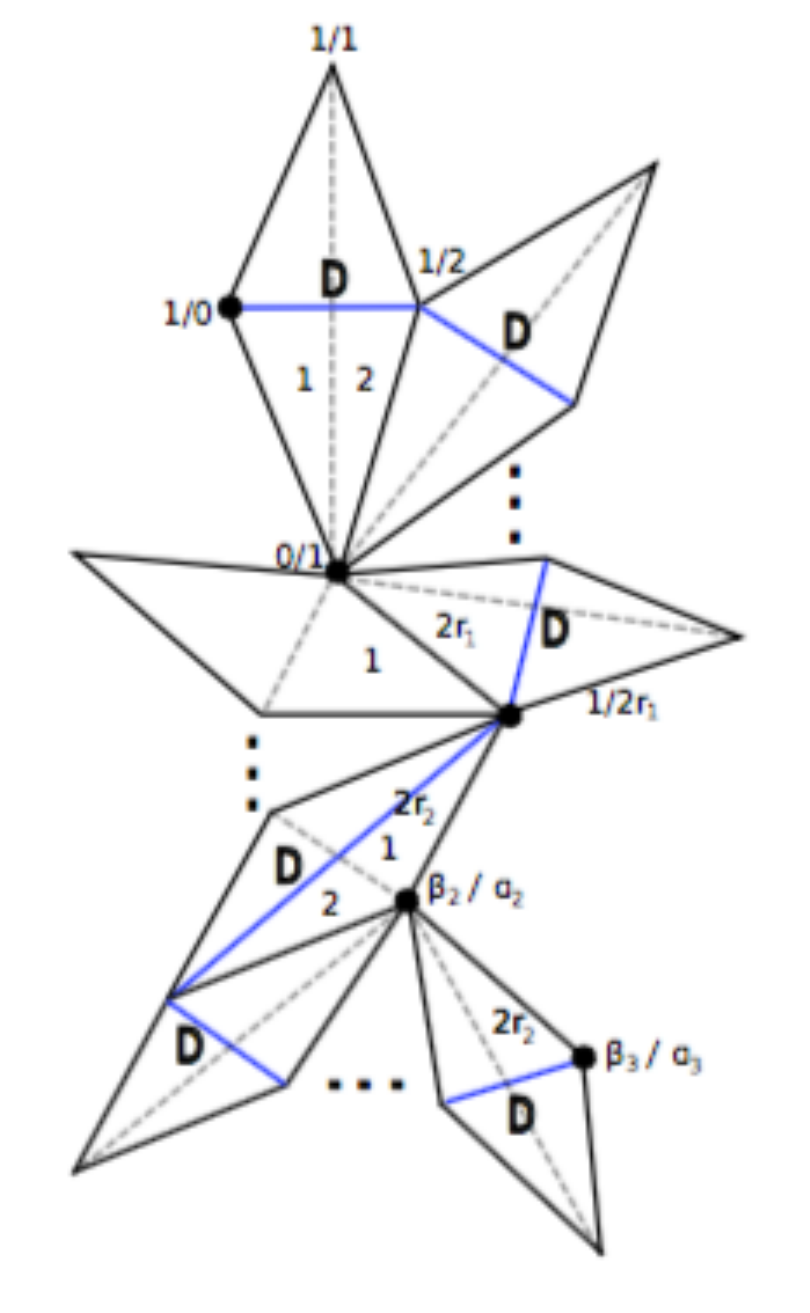}
        \caption{$2r_2$ positive and $2r_3$ negative}
         \label{pathD-2}
   \end{subfigure}
   \begin{subfigure}[b]{0.4\textwidth}
        \includegraphics[width=\textwidth]{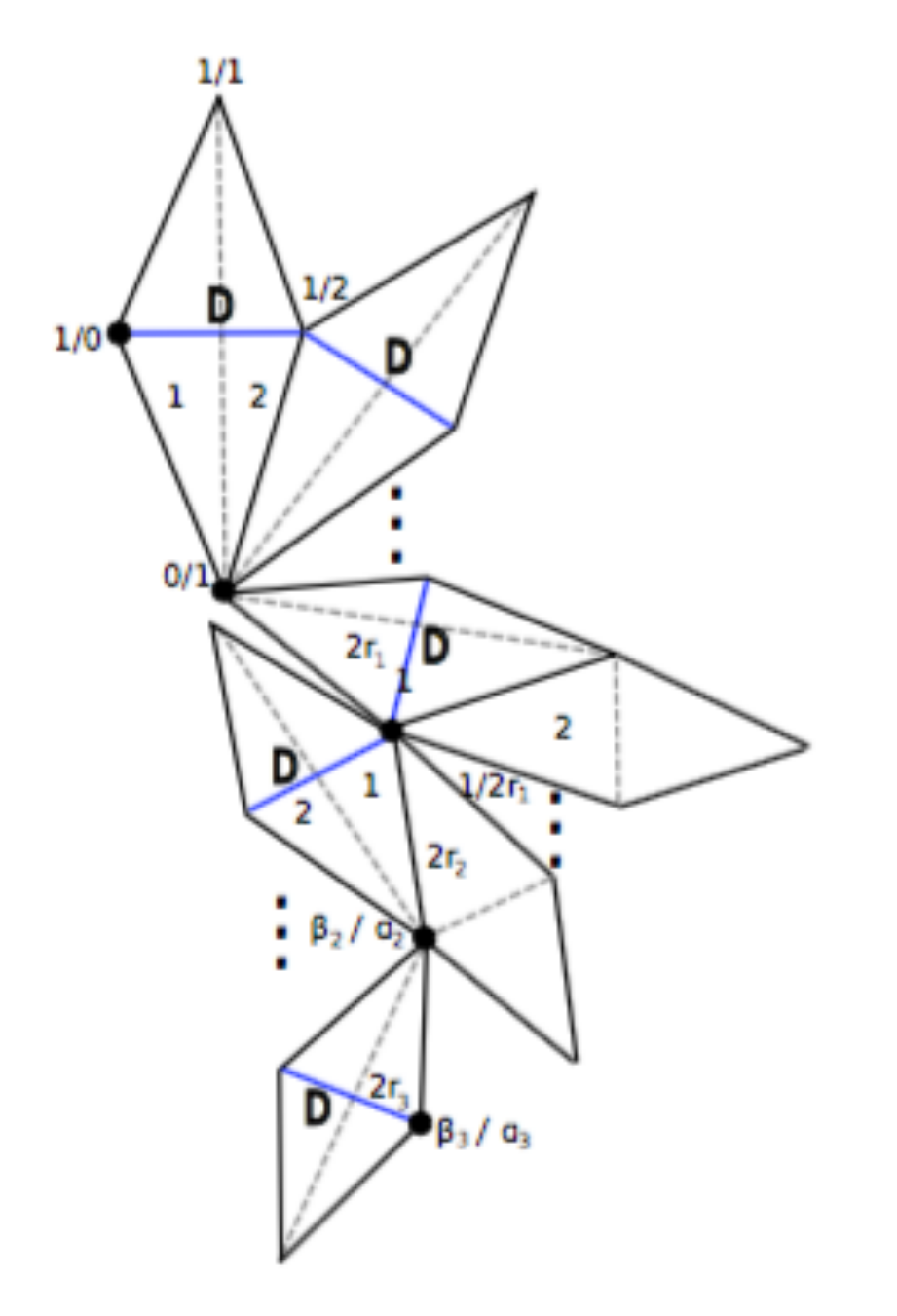}
        \caption{$2r_2$ and $2r_3$ negative}
         \label{pathD-3}
   \end{subfigure}
   \begin{subfigure}[b]{0.4\textwidth}
        \includegraphics[width=\textwidth]{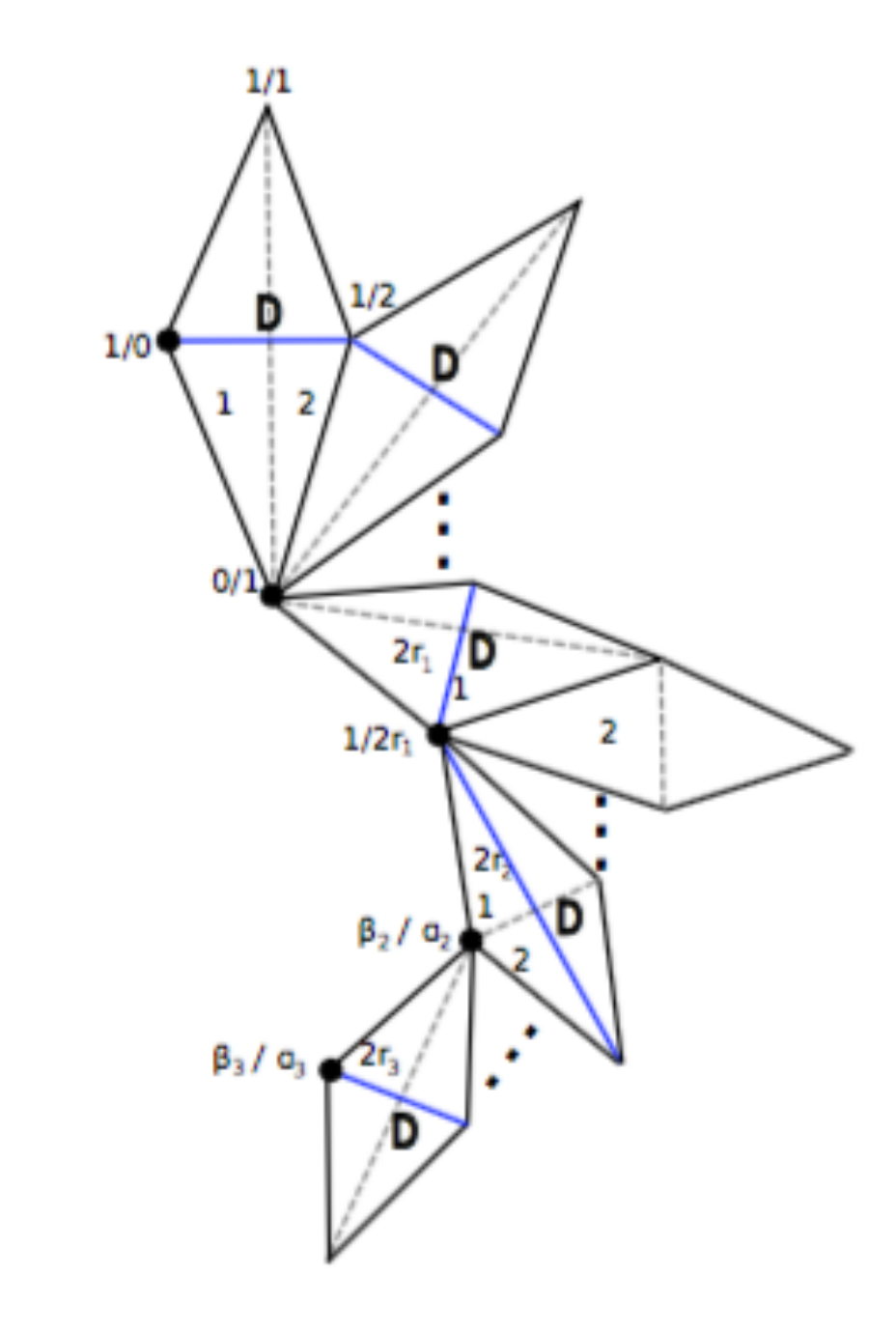}
        \caption{$2r_2$ negative and $2r_3$ positive}
         \label{pathD-4}
   \end{subfigure}
   \caption{Analyse of $D$-edge-path in $D_0$ to belong to a string.}
\end{figure}

\begin{corollary}
 Let $L_{\beta/\alpha}=K_1 \cup K_2$ be a link with $lk(K_1, K_2)=0$. A surface $S$ in $S^3-L_{\beta/\alpha}$ associated to a $D$-edge-path is not a fiber of a fibering $S^3-L_{\beta/\alpha} \rightarrow S^1$.
\end{corollary}
\begin{proof}
The third part of Theorem \ref{fiberingtheorem} implies that if the the surface $S$ is carried by a $D$-edge-path, then the continued fraction expansion for $\beta/\alpha$ is of the form $[0; 2r_1,-2, 2r_2, ... -2, 2r_n ]$ with $2r_i$ positive for all $i$. Thus the linking number is not equal to zero, a contradiction.
\end{proof}

\section{Applications}
\label{applications}
In this section we compute the genus of tunnel number one satellite knot, as well as torti-rational knots. Hirasawa and Murasugi, \cite{HM} have computed the genus of such knots using algebraic techniques, namely the Alexander polynomial. We give criteria to determine fiberedness of satellite tunnel number one knots only when $lk(K_1, K_2)\neq 0$.

\subsection{Tunnel number one satellites knots}
Morimoto and Sakuma \cite{MS} determined the knot types of satellite tunnel number one knots in $S^3$. These knots are constructed as follows. Let $K_0$ be a  $(p,q)$-torus knot in $S^3$ with $p\neq 1$ and $q\neq 1$, and let $L_{\beta/ \alpha}= K_1 \cup K_2$ be a 2-bridge in $S^3$ with $\alpha \geq 4$. Note that $K_0$ is a non-trivial knot, and $L_{\beta/ \alpha}$ is neither a trivial link nor a Hopf link. Since $K_1$ is the trivial knot in $S^3$, there is a an orientation preserving homeomorphism $f: E(K_1) \rightarrow N(K_0)$ which takes a meridian $m_2 \subset \bd E(K_1)$ of $K_1$ to a fiber $h \subset \bd N(K_0)=\bd E(K_0)$ of the unique Seifert fibration of $E(K_0)$. The knot $f(K_2)\subset N(K_0) \subset S^3$ is denoted by the symbol $K(\alpha, \beta; p, q)$. Every satellite knot of tunnel number one has the form $K(\alpha, \beta; p, q)$ for some integers $\alpha, \beta, p, q$. Eudave-Mu\~noz \cite{EM} obtained another description of these knots.

Let $l$ and $m$ be a preferred longitude and a meridian for $\bd N(K_0)$, respectively. 
Notice that $\Delta(l, h)=pq$ and then $\Delta(f^{-1}(l), m_2)=pq$,  where $\Delta$ stands for the  geometric intersection of two curves.

The next lemma can be found in \cite{BZ}, and it will be useful.

\begin{lemma}
\label{BZ}
Let $K=K(\alpha, \beta; p, q)$ be a satellite tunnel number one knot. Let $F$ be a minimal genus Seifert surface for $K$. The surface $F$ can be isotoped in such a way that $F \cap \bd N(K_0)$ consists of $\vert lk(K_1, K_2)\vert$ preferred longitudes and $F \cap (S^3 -N(K_0))$ is made of $\vert lk(K_1, K_2)\vert$ components which are Seifert surfaces for $K_0$.
\end{lemma}

First we consider the case when $lk(K_1, K_2)=0$.

\begin{theorem}
When the linking number is zero, the genus of a satellite tunnel number one knot is one half the wrapping number of $K_2$ in $E(K_1)$.
\end{theorem}

Suppose the 2-bridge presentation of $L_{\beta/ \alpha}$ is given relative to some 2-sphere $S$ in $S^3$ bounding 3-balls $W_0, W_1$ such that $L$ intersects $S$ transversely and $L\cap W_i$ is a disjoint union of two arcs. 
Consider $S\times I$ be a product regular neighborhood of $S$ in $S^3$, and let $h: S\times I \rightarrow I$ be the height function. We denote the level surfaces $h^{-1}(r)= S \times \{r\}$ by $S_r$ for each $0\leq r \leq 1$. $S_0$ bounds a 3-ball  $H_0$, and $S_1$ bounds a 3-ball $H_1$, such that $S^3= H_0 \cup (S\times I) \cup H_1$. Assume that $S_0 \subset W_0$, $S_1\subset W_1$, and that $h\vert (S\times I) \cap L$ has no critical points (so $(S\times I) \cap L_{\beta/ \alpha}$ consists of monotone arcs).

Let $F$ be an essential surface properly embedded in the exterior $E(L_{\beta/ \alpha})=S^3 - int N(L)$. 

By general position, an essential  surface can always be isotoped in $E(L_{\beta/ \alpha})$ so that:
-
\begin{description}
\item[(M1)] $F$ intersects $S_0 \cup S_1$ transversely; we denote the surfaces $F \cap H_0$, $F \cap H_1$, $F \cap (S\times I)$ by $F_0$, $F_1$, $\tilde{F}$, respectively;
\item[(M2)] each component of $\bd F$ is either a level meridian circle of $\bd E(L_{\beta/ \alpha})$ lying in some level set $S_r$ or it is transverse to all the level meridians circles of $\bd E(L_{\beta/ \alpha})$ in $S \times I$;
\item[(M3)]for $i=0,1$, any component of $F_i$ containing parts of $L$ is a cancelling disk for some arc of $L_{\beta/ \alpha} \cap H_i$; in particular, such cancelling disks are disjoint from any arc of $L_{\beta/ \alpha}\cap H_i$ other than the one they cancel;
\item[(M4)] $h\vert \tilde{F}$ is a Morse function with a finite set $Y(F)$ of critical points in the interior of $\tilde{F}$, located at different levels; in particular, $\tilde{F}$ intersects each noncritical level surface transversely.
\end{description}

We define the complexity of any surface satisfying $(M1)-(M4)$ as the number
\begin{center}
$c(F)= \vert \bd F_0 \vert +  \vert \bd F_1 \vert + \vert Y(F) \vert$,
\end{center}

\noindent where $\vert Z \vert$ stands for the number of elements in the finite set $Z$, or the number of components of the topological space $Z$.

We say that $F$ is \textit{meridionally incompressible} if whenever $F$ compresses in $S^3$ via a disk $D$ with $\bd D= D \cap F$ such that $D$ intersects $L_{\beta/ \alpha}$ in one point interior to $D$, then $\bd D$ is parallel in $F$ to some boundary component of $F$ which is a meridian circle in $\bd E(L_{\beta/ \alpha})$; otherwise, $F$ is \textit{meridionally compressible}. Observe that if $F$ is essential and meridionally compressible then a meridional surgery on $F$ produces a new essential surface in $E(L_{\beta/ \alpha})$.

The following is Lemma 3.2 of \cite{RV}.

\begin{lemma}
\label{KT}
Let $F$ be a surface in $S^3$ spanned by $K_2$ (orientable or not) and transverse to $K_1$, such that $F'= F \cap E(L_{\beta/ \alpha})$ is essential and meridionally incompressible in $E(L_{\beta/ \alpha})$. If $F'$ is isotoped so as to satisfy (M1)-(M4) with minimal complexity, then $\vert Y(F')\vert= 2- (\chi(F')+ \vert \bd F' \vert) $, and
\begin{enumerate}
\item each critical point of $h\vert \tilde{F}$ is a saddle,
\item for $0\leq r \leq 1$ any circle of $S_r \cap F$ is  nontrivial in $S_r- L_{\beta/ \alpha}$ and $F$, and
\item $F_0$ and $F_1$ each consists of one cancelling disk.
\end{enumerate}
\end{lemma}

When $lk(K_1, K_2)=0$, Lemma \ref{BZ} implies that $F'=f^{-1}(F) \subset E(K_1)$. Moreover $F'$ is an incompressible genus $g$ Seifert surface for $K_2$.

\begin{lemma}
\label{meri-comp}
The surface $F'$ can be meridionally compressed $g$-times to obtain a  disk $\Sigma$ that satisfies the conditions of Lemma \ref{KT}. And $g$ is equal to the one half the wrapping number of $K_2$ with respect to $E(K_1)$. Moreover, if 
$[s;2r_1,...,2r_k]$ is the continued fraction expansion for $\beta/\alpha$ with $s=0$ or $1$ such that  $k$ odd, the genus of $K(\alpha, \beta, p, q)$ is $\Sigma \vert r_i\vert$.
\end{lemma}
\begin{proof}
We will proceed by induction on the pair $(g(F'), \vert Y(F')\vert)$. By Lemma 21 of  \cite{EMR} we know that a surface $S$ with $(g(S), \vert Y(S)\vert) \leq (2,4)$ meridionally compresses $g(S)$-times to a disk satifiying Lemma \ref{KT}.
Let us assume that the result is true for any  surface $S$ with  $(g(S), \vert Y(S)\vert) \leq (g(F'), \vert Y(F')\vert)$. Suppose that $F'$ is meridionally incompressible, we can apply Lemma \ref{KT}, and using the same arguements in Lemma 21 of \cite{EMR}, we  obtain a contradiction and thus $F'$ must be meridionally compressible. Moreover after performing the meridional compression a connected surface $F^2$ is obtained, and $g(F^2)= g(F')-1$ and  $\vert Y(F^2) \vert= \vert Y(F') \vert -2$. By  induction hypothesis $F^2$ compresses meridionally $g(F^2)$-times to a disk satisfying Lemma \ref{KT}. But $F^2$ was obtained by compression $F'$ once, thus $F'$ compresses meridionally $g(F')$-times to the required disk $\Sigma$. Thus $K_2$ spans $\Sigma$ which intersects meridionally $K_1$ in $2g(F')$ points, this implies that the wrapping number of $K_2$ in the solid torus $E(K_1)$ is equal to $2g(F')$. Now, to recover $F'$ from $\Sigma$ we must attached $g(F')$ tubes, therefore the last part of the statement is true.   
\end{proof}

Next we consider the case when $lk(K_1, K_2)\neq 0$.

Let $l$ and $m$ be a preferred longitude and a meridian for $\bd N(K_0)$, respectively. 
Notice that $\Delta(l, h)=pq$ and then $\Delta(f^{-1}(l), m_2)=pq$.
 
Let $F$ be a minimal genus Seifert surface for $K=K(\alpha, \beta; p, q)$, by Lemma \ref{BZ} the surface $F$ can be isotoped in such a way that $F \cap \bd N(K_0)$ consists of $\vert lk(K_1,K_2)\vert$ preferred longitudes and $F \cap (S^3 -N(K_0))$ is made of $\vert lk(K_1,K_2)\vert$ components which are Seifert surfaces for $K_0$. Let $\tilde F=F \cap N(K_0)$, notice that once we determine the genus of $\tilde F$ the genus of $F$ is obtained by adding $\vert lk(K_1, K_2)\vert$-times $(\vert p \vert-1)(\vert q\vert -1)/2$, which is the genus of the torus knot $K_0$.

The surface $F'= f^{-1}(\tilde F)$ is an incompressible surface spanned by $L_{\beta/ \alpha}=K_1\cup K_2$ whose boundary consists of one component in $K_2$ and $\vert lk(K_1, K_2)\vert$-boundary components in $K_1$.

\begin{lemma}
\label{lemaslopes}
The boundary slope of surface $F'$ in $K_2$ equals $-lk(K_1, K_2)^2pq$ and the boundary slope of $F'$ in $K_1$ equals $\frac{-1}{pq}$.
\end{lemma}
\begin{proof}
Let $l_1$ and $m_1$ be the standard longitude and meridian of $K_1$ (chose any orientation of $K_1$) and let $\lambda$ and $\mu$ the longitude and meridian of $K_0$, the morphism $f: \partial E(K_1) \to \partial N(K_0)$ sends $m_1$ to $pq\mu + \lambda$ (which is the fiber of the Seifert fibration of $E(k_0)$) and $l_1$ to $\mu$ so, the longitude $\lambda$ is identified with $-pql_1 + m_1$ this means that the slope of $F'$ in $K_1$ is equals to $\frac{-1}{pq}$.

Let $\partial_2F'$ be the boundary of $F'$ on $K_2$ and $\partial_1F'$ be the one on $K_1$. It follows that $\partial_2F'$ is homological equivalent to $\partial_1F'$ on $E(L_{\beta/\alpha})$. Observe that the inclusion $\partial N(K_2) \to E(L_{\beta/\alpha})$ induces an injection between the first homological groups, so $\partial_1F'$ would be equivalent to only one class on $H_1(\partial N(K_2))$; that has to be $\partial_2F'$.

Now, let $l_2$ and $m_2$ be the standard longitude and meridian of $K_2$ and $lk = lk(K_1,K_2)$. In $E(L_{\beta/\alpha})$, $l_2$ is homological equivalent to $lk \cdot m_1$ (consider the disk bounded by $l_2$) and also $l_1$ is homological equivalent to $lk \cdot m_2$. Then, $\partial_1 F' \sim \partial_1 F' \sim lk\cdot (-pql_1 + m_1) = -pq\cdot lk \cdot l_1 + lk \cdot m_1 = -pq \cdot lk^2 \cdot m_2 + l_2 $, this implies that the boundary of $F'$ in $K_2$ is homological equivalent to $ -pq\cdot lk^2 \cdot m_2 + l_2$ i.e. its slope is $- pq \cdot lk^2$

\end{proof}

In order to find the minimal genus of $K=K(\alpha, \beta; p, q)$, first we need to determine the minimal genus of the surface $F'$ for the rational link $L_{\beta/\alpha}=K_1 \cup K_2$ with the above characteristics. That is to say, a surface $F'$ with one boundary component on $\partial N(K_2)$ and $\vert lk(K_1, K_2)\vert$-boundary components on $\partial N(K_1)$, with boundary slopes as in Lemma \ref{lemaslopes}, i.e, $\rho=1$ and $\mu=\vert pqlk(K_1, K_2)\vert$. Since $p, q \neq 1$, then $\mu\neq 1$ even if $\vert lk(K_1, K_1)\vert= 1$. Observe that if $pq \geq 0$ then the boundary slopes turned out to be negative, and if $pq \leq 0$ they are positive. In both cases, the path associated to the continued fraction expansion $[r; 2r_1,..., 2r_k]$ for $\beta/\alpha$, with $r=0$ or $1$ and $k$-odd, consists only of $A$ and $D-$type edges by Lemma \ref{lemma:sillas}. By Proposition \ref{genusAD} it is possible to compute the genus of the orientable surface carried by such path. Moreover, when $r=0$ the corresponding continued fraction is the one that gives rise to the surface with negative boundary slopes in both components of $L_{\beta/\alpha}$, by Corollary \ref{linkingnegativo}. When $r=1$ we obtain a surface with positive boundary slopes on both components of $L_{\beta/\alpha}$, by Corollary \ref{linkingpositivo}. Summarizing we have the following result.

\begin{theorem}
Let $L_{\beta/\alpha}=K_1 \cup K_2$ be the 2-bridge link given by the tunnel number one satellite knot $K(\alpha,\beta, p, q)$. Suppose $lk(K_1, K_2)\neq 0$. Then
\begin{enumerate}
\item If $0\leq \beta \leq \alpha$, $pq\geq 0$ and $[0;2n_1,...,2n_j]$ is the unique continued fraction for $\beta/\alpha$ with $j$ odd, the genus of $F'$ is:
\begin{equation*} 
\frac{1}{2} \Bigg[\displaystyle{\left(-1 + \sum_{ k \; odd}|n_k|\right) } \; (|lk(K_1,K_2)pq| - 1) + (j + 1) - (\vert lk(K_1, K_2)\vert+1)\Bigg]
\end{equation*}
where $k\in \{1, ..., j\}$
\item If $0\leq \beta \leq \alpha$, $pq\leq 0$ and $[1;2m_1,...,2m_i]$ is the unique continued fraction for $\beta/\alpha$ with $i$ odd, the genus of $F'$ is:
\begin{equation*} 
\frac{1}{2} \Bigg[\displaystyle{\left(-1 + \sum_{ h \; odd}|m_h|\right) } \; (|lk(K_1,K_2)pq| - 1) + (i + 1) - (\vert lk(K_1, K_2)\vert+1)\Bigg]
\end{equation*}
where $h\in \{1, ..., i\}$
\end{enumerate}
\end{theorem}

\begin{corollary}
Let $K=K(\alpha,\beta, p, q)$ be a tunnel number one satellite knot, the genus of $K$ is:
\begin{equation*}
g(K)= g(F')+ \vert lk(K_1, K_2)\vert \frac{(\vert p\vert-1)(\vert q \vert-1)}{2}
\end{equation*}

\end{corollary}

We can also determined when a satellite tunnel number one knot $K=K(\alpha,\beta, p, q)$ is fibered, if $lk(K_1, K_2)\neq 0$. Recall that the $(p,q)$-torus knot $K_0$ is fibered. A surface $F$ for $K$ is broken into pieces: $\tilde{F}=F\cap \partial N(K_0)$ and $\vert lk(K_1, K_2)\vert$ components which are Seifert surfaces for $K_0$. These pieces are glued along a fiber of the Seifert fibration of the knot $K_0$. Thus, if $F'=f^{-1}(\tilde F)$ is a fiber of a fibering of $S^3-L_{\beta/\alpha}\rightarrow S^1$ then $F$ will be a fiber of a fibering $S^3-K \rightarrow S^1$. Theorem \ref{fibering} part (1) gives us the condition to recognize when $F'$ is a fiber for $L_{\beta/\alpha}$.

\begin{proposition}
A tunnel number one satellite knot $K(\alpha,\beta, p, q)$, where $lk(K_1, K_2)\neq 0$, is fibered if and only if $\beta/\alpha$ has a continued fraction expansion or type $[r; 2r_1,2\epsilon_1, 2r_3, ...,2\epsilon_k, 2r_k]$, with $r=0$ or $1$, $\vert \epsilon\vert=1$ and $k$-odd. 
\end{proposition}

\subsection{Torti-rational knots}
Let $L_{\beta/ \alpha}= K_1 \cup K_2$ be a 2-bridge in $S^3$. Since $K_1$ is a trivial knot in $S^3$, $K_2$ can be considered as a knot in an unknotted solid torus $V$ and $K_1$ a meridian of $V$. Then by applying Dehn twists along $K_1$ in an arbitrary number of times, say $r$, we obtain a new knot $K$ from $K_2$. We call this knot a torti-rational knot and it is denoted by $K(\beta/\alpha; r)$, in particular it is contained in $V$. Let $F$ be a minimal genus Seifert surface for $K(\beta/\alpha; r)$ of genus $g$. Consider the case when $lk(K_1,K_2)=0$, we need a result that shows $F\subset V$, and this will let us compute the genus of $F$ as in the case of satellite tunnel number one knots.

\begin{lemma}
Let $F$ be a minimal genus Seifert surface for the torti-rational knot $K(\beta/\alpha; r)$. Suppose $lk(K_1, K_2)=0$, then $F\subset V$.
\end{lemma}
\begin{proof}
Assume that $F\cap \bd V \neq  \emptyset$, $F$ can be isotoped to intersect $\bd V$ in $n$-longitudes and $F \cap (S^3-V)$ consisting of $n$-disjoint disks. Let $\tilde{F}=F\cap V$, after undoing the $r$-Dehn twists along $K_1$, an essential spanning surface $F'$ for $K_1\cup K_2$ is obtained. The surface $F'$ has one boundary component $\bd_2 F'$ parallel to $K_2$ and $n$- boundary components $\bd_1 F'$ of slope $1/r$. Lemma \ref{lemma:signsum-and-linking} states that $\vert 1/r \Sigma_v\vert= \vert lk(K_1, K_2)\vert$, then we have that $\vert \Sigma_v\vert=0$. In particular the boundary components of $F$ along $K_1$ have different orientations. Lemma \ref{lemma:sillas} implies that if $\mu>1$ and if a $B$-type saddle occurs then $\vert \Sigma_v \vert =1$, which is a contradiction. Or if a $D$-type saddle appears then all boundary components of $F'$ have the same orientation, which is not true. If $\mu=1$ then $\vert \Sigma_v \vert =1$, but it equals zero. Thus $\mu=0$, implying that $F'$ does not have boundary components on $K_1$, applying the $r$-Dehn twist we recover $F$ which is contained in $V$. 

\end{proof}

Similarly to Lemma \ref{meri-comp}, the surface $F'$ can be compressed meridionally $g$-times to obtain a disk satisfying the conditions of Lemma \ref{KT}. Thus we have the following result.

\begin{proposition}
Let $F$ be minimal Seifert genus surface for the torti-rational knot $K(\beta/\alpha; r)$ such that $lk(K_1,K_2)=0$. The genus $g$ of $F$ is equal to one half the wrapping number of $K_2$ with respect to $E(K_1)$.
\end{proposition} 

Now consider the case $lk(K_1, K_2)\neq 0$, then $F\cap \bd V \neq \emptyset$.
We will determine the genus of $F$ in terms of the parameters $\beta, \alpha, r$ and $lk(K_1, K_2)$.

\begin{theorem}
Let $K(\beta/\alpha; r)$ be a torti-rational knot and $F$ a minimal genus Seifert surface for it. Suppose that $lk(K_1, K_2)\neq 0$. Then:
\begin{enumerate}
\item If $r > 1$ and $[1;2m_1,...,2m_i]$  is the unique continued fraction for $\beta/\alpha$ with $i$ odd, the genus of $F$ is:
\begin{equation*} 
\frac{1}{2} \Bigg[\displaystyle{\left(-1 + \sum_{ h \; odd}|m_h|\right) } \; (|lk(K_1,K_2)r| - 1) + (i + 1) - (\vert lk(K_1, K_2)\vert+1)\Bigg]
\end{equation*}
where $h\in \{1, ..., i\}$

\item If $r < 1$ and $[0;2n_1,...,2n_j]$ is the unique continued fraction for $\beta/\alpha$ with $j$ odd, the genus of $F$ is:
\begin{equation*} 
\frac{1}{2} \Bigg[\displaystyle{\left(-1 + \sum_{ k \; odd}|n_k|\right) } \; (|lk(K_1,K_2)r| - 1) + (j + 1) - (\vert lk(K_1, K_2)\vert+1)\Bigg]
\end{equation*}
where $k\in \{1, ..., j\}$

\item If $\vert r \vert  > 1 $ and $\vert lk(K_1,K_2)\vert > 1$. Let  
$[s;2r_1,...,2r_k]$  be the continued fraction expansion for $\beta/\alpha$ with $s=0$ or $1$ such that  $k\geq 3$ and  $\vert r_t \vert \geq 2$ for all $t$. The genus of $F$ is:
\begin{equation*}
    1+ \frac{(\vert lk(K_1, K_2)\vert+1)(k-3)}{4}
\end{equation*}

\item If $\vert r \vert =1$ and $\vert lk(K_1,K_2)\vert = 1$ and $[0;2n_1,...,2n_j]$ and $[1;2m_1,...,2m_i]$ are the continued fraction for $\beta/\alpha$ with $j,i$ odd. The genus of $F$ is:
\begin{equation*}
    min \Bigg(\frac{i-1}{4}, \frac{j-1}{4}\Bigg)
\end{equation*}

\end{enumerate}
\end{theorem}

\begin{proof}
The surface $F$ can be isotoped to intersect $\bd V$ in $n$-longitudes and $F \cap (S^3-V)$ consisting of $n$-disjoint disks. Let $\tilde{F}=F\cap V$, after undoing the $r$-Dehn twists along $K_1$, an essential spanning surface $F'$ for $K_1\cup K_2$ is obtained. The surface $F'$ has one boundary component $\bd_2 F'$ parallel to $K_2$ and $n$- boundary components $\bd_1 F'$ of slope $1/r$. If we determine the genus of $F'$ it will be the genus of $F$. By performing the corresponding $r$-Dehn twists along $K_1$ we recover $\tilde{F}$, after capping of the $n$-boundary components of $\tilde{F}$ we have $F$, thus $F$ and $F'$ have the same genus.

For the essential surface $F'$, $\rho=1$ and $\mu= \vert r \vert n$. By the formula of Lemma \ref{lemma:signsum-and-linking} we get $n=\vert lk(K_1, K_2)\vert$. The surface $F'$ corresponds to some edge-path $\gamma$ on a $D_t$ diagram, since $n,lk(K_1, K_2), \rho \neq 0$ then $t\neq 0, \infty$. If $\mu > 1$, Corollary \ref{coro:AB-and-AD-edgepaths} implies that $\gamma$ is either an $AD$-edge-path or an $AB$-edge-path.

Suppose  $r > 1$, the boundary components $\bd_1 F'$ have positive slope $1/r$, thus the slope is in correspondence with the slope given by the surface defined by the continued fraction expansion $[1;2m_1,...,2m_i]$ for $\beta/\alpha$, by Corollary \ref{linkingpositivo}. Applying Proposition \ref{genusAD} we obtain the result claimed in $(1)$ .

Similarly, if $r < 1$ the boundary slopes of $\bd_1 F'$ are negative and by Corollary \ref{linkingnegativo}, $F'$ is in correspondence with the path given by the continued fraction expansion $[0;2n_1,...,2n_j]$. The genus of $F'$ is given by Proposition \ref{genusAD} and hence we have proof $(2)$.

If $\vert r \vert  > 1 $, then $\mu= n$. If $\vert lk(K_1, K_2)\vert >1$; then $\gamma$ is a minimal $AB$-edge-path. Let $[s;2r_1,...,2r_k]$ be continued fraction expansion for $\beta/\alpha$ with $s=0$ or $1$ and such that $k\geq3$ and $r_t \geq 2$ for all $t$. The genus of $F'$ is computed using  Proposition \ref{genusAB}. We have proved $(3)$.

In the case that $\vert r \vert =1$ and $\vert lk(K_1, K_2)\vert =1$, the path $\gamma$ is an $A$-edge-path. Let $[0;2n_1,...,2n_j]$ and $[1;2m_1,...,2m_i]$ be the continued fraction for $\beta/\alpha$ with $j,i$ odd. Using Proposition \ref{genusA}, we can compute the genus of the two surfaces corresponding to the both continued fractions. We pick the minimum between them, and we get part $(4)$ of the Theorem.

\end{proof}


\begin{thebibliography}{10}

\bibitem{BZ}
Gerhard Burde, Heiner Zieschang, and Michael Heusener.
\newblock {\em Knots}, volume~5 of {\em De Gruyter Studies in Mathematics}.
\newblock De Gruyter, Berlin, extended edition, 2014.

\bibitem{EMR}
M.~Eudave-Mu\~noz, F.~Manjarrez-Guti\'errez, and E.~Ram\'irez-Losada.
\newblock Classification of genus two knots which admit a
  {$(1,1)$}-decomposition.
\newblock {\em Topology Appl.}, 230:468--489, 2017.

\bibitem{EM}
Mario Eudave-Mu\~noz.
\newblock On nonsimple {$3$}-manifolds and {$2$}-handle addition.
\newblock {\em Topology Appl.}, 55(2):131--152, 1994.

\bibitem{FH}
W.~Floyd and A.~Hatcher.
\newblock The space of incompressible surfaces in a {$2$}-bridge link
  complement.
\newblock {\em Trans. Amer. Math. Soc.}, 305(2):575--599, 1988.

\bibitem{GHS}
Hiroshi Goda, Chuichiro Hayashi, and Hyun-Jong Song.
\newblock Dehn surgeries on 2-bridge links which yield reducible 3-manifolds.
\newblock {\em J. Knot Theory Ramifications}, 18(7):917--956, 2009.

\bibitem{HT}
A.~Hatcher and W.~Thurston.
\newblock Incompressible surfaces in {$2$}-bridge knot complements.
\newblock {\em Invent. Math.}, 79(2):225--246, 1985.

\bibitem{HM}
Mikami Hirasawa and Kunio Murasugi.
\newblock Fibered torti-rational knots.
\newblock {\em J. Knot Theory Ramifications}, 19(10):1291--1353, 2010.

\bibitem{HS}
Jim Hoste and Patrick~D. Shanahan.
\newblock Computing boundary slopes of 2-bridge links.
\newblock {\em Math. Comp.}, 76(259):1521--1545, 2007.

\bibitem{L}
Alan~Eliot Lash.
\newblock {\em Boundary curve space of the {W}hitehead link complement}.
\newblock ProQuest LLC, Ann Arbor, MI, 1993.
\newblock Thesis (Ph.D.)--University of California, Santa Barbara.

\bibitem{Ma}
Hiroshi Matsuda.
\newblock Genus one knots which admit {$(1,1)$}-decompositions.
\newblock {\em Proc. Amer. Math. Soc.}, 130(7):2155--2163, 2002.

\bibitem{MS}
Kanji Morimoto and Makoto Sakuma.
\newblock On unknotting tunnels for knots.
\newblock {\em Math. Ann.}, 289(1):143--167, 1991.

\bibitem{RV}
Enrique Ram\'irez-Losada and Luis~G. Valdez-S\'anchez.
\newblock Crosscap number two knots in {$S^3$} with {$(1,1)$} decompositions.
\newblock {\em Bol. Soc. Mat. Mexicana (3)}, 10(Special Issue):451--465, 2004.

\end{thebibliography}


\begin{thebibliography}{99}
\bibitem{BZ} Gerhard Burde and Heiner Zieschang, \textit{Knots}, de Gruyter Studies in Mathematics (1985).

\bibitem{EM} Mario Eudave-Mu\~noz, \textit{On nonsimple 3-manifolds and 2-handle addition}, Topology Appl. 55 (1994), no. 2, 131-152.

\bibitem{EMR} Mario Eudave-Mu\~noz, Fabiola Manjarrez-Guti\'errez and Enrique Ram\'{i}rez-Losada, \textit{Classification of genus two which admit a $(1,1)$-decomposition}, Top.Appl. (2017), vol 230, 468-489.


\bibitem{FH} W. Floyd, A. Hatcher, \textit{The space of incompressible surfaces in a 2-bridge link complement}, Trans. Amer. Math. Soc. 305 (1988), no. 2, 575-599.

\bibitem{GHS} Hiroshi Goda, Hayashi Chuichiro and Song Hyun-Jong, \textit{Dehn surgeries on 2-bridge links which yield reducible
              3-manifolds}, J. Knot Theory and its Ramifications (2009), vol. 18, 917-956.

\bibitem{HT} Allen Hatcher and William Thurston,  \textit{Incompressible surfaces in {$2$}-bridge knot complements}. Inventiones Mathematicae. (1985) vol. 79, no. 2, 225-246.


\bibitem{HM} Mikami Hirasawa and Kunio Murasugi, \textit{Fibered torti-rational knots}, J. Knot Theory and Its Ramifications, (2010) vol. 19, no. 10, 1291-1353.   

\bibitem{HS} Jim Hoste and Patrick Shanahan, \textit{Computing boundary slopes of 2-bridge links}, Math. Comp. (2007) vol 76, no. 259, 1521-1545. 

\bibitem {Ma} Hiroshi Matsuda, \textit{Genus one knots which admit $(1,1)$-decompositons}, Proc. Amer. Math. Soc. 130 (2002), no. 7, 2155-2163.

\bibitem{L} Alan Lash, \textit{Boundary curve space of the Whitehead link complement},. Thesis PhD. University of California, Santa Barbara.1993,

\bibitem{MS} Kanji Morimoto and Makoto Sakuma, \textit{On unknotting tunnels for knots}, Math. Ann. 289 (1991), no. 1, 143-167.

\bibitem {RV} E. Ram\'{i}rez-Losada and  L. G. Valdez-S\'anchez, \textit{Crosscap number two knots in $S^3$ with $(1,1)$ decompositions}, Bol. Soc. Mat. Mexicana 3 (2004), vol. 10, special issue, 451-465.





\end{thebibliography}
\end{document}